\documentclass[reqno,oneside,11pt,letterpaper]{amsart}

\usepackage{dwpreamble}
\usepackage{dwcommands}

\usepackage[normalem]{ulem} 

 \newcommand{\semi}{\hbar}
 \newcommand{\Psisch}{\Psi_{\scl,\semi}}

\begin{document}

\title[Dirac operators and local invariants on perturbations of Minkowski space]{Dirac operators and local invariants on perturbations of Minkowski space}

\author{Nguyen Viet Dang}
\address{IRMA, Université de Strasbourg, Strasbourg, France \vspace{-0.3cm}}
\address{Institut Universitaire de France, Paris, France}
\email{nvdang@unistra.fr}
\author{Andr\'as \textsc{Vasy}}
\address{Department of Mathematics, Stanford University, Stanford CA, USA}
\email{andras@stanford.edu}
\author{Micha{\l} \textsc{Wrochna}}
\address{Mathematical Institute, Universiteit Utrecht, Utrecht, The Netherlands \vspace{-0.3cm}} \address{Department of Mathematics \& Data Science, Vrije Universiteit Brussel, Brussels, Belgium}
 \email{{m.wrochna@uu.nl}}
\keywords{wave equation, microlocal analysis, spectral zeta functions, Dirac operators}

\begin{abstract} For small perturbations of Minkowski space, we show that the square of the Lorentzian Dirac operator $P= -\D^2$ has real spectrum apart from possible poles in a horizontal strip. Furthermore, for $\varepsilon>0$ we relate the poles of the spectral zeta function  density of $P-i\varepsilon$ to local  invariants, in particular to the Lorentzian scalar curvature. The proof involves microlocal propagation and radial estimates in a resolved  scattering calculus as well as high energy estimates in a further resolved classical-semiclassical calculus.
\end{abstract}

\maketitle

\section{Introduction}

\subsection{Introduction and main result} The spectral theory of geometric differential operators on a Lorentzian manifold $(M^\circ,g)$ is an emerging topic with surprising features. In spite of   non-ellipticity, the Laplace--Beltrami or wave operator $\square_g$ has been shown to be essentially self-adjoint in a variety of settings, including static spacetimes \cite{derezinski}, asymptotically Minkowski spacetimes  \cite{vasyessential,nakamurataira,Nakamura2022,JMS}, and Cauchy-compact asymptotically static spacetimes \cite{Nakamura2022a}, see also \cite{Tadano2019,taira,kaminski,Taira2020a,cdv,Taira2022,Wrochna2022,Derezinski2024} for other results on spectral properties of $\square_g$ and the limiting absorption principle and \cite{GHV,Vasy2017b,GWfeynman,Gerard2019b,vasywrochna,MolodhykVasy,Derezinski2024} for the closely related subject of Feynman propagators. Furthermore,  at least in the asymptotically Minkowski (or ultra-static) case,  it turns out to be possible to define a spectral zeta  function density as the meromorphic extension of the trace density
$$
(\square_g - i \varepsilon)^{-\alpha}(x,x)\in \cf(M^\inti),
$$
initially defined for large $\Re \alpha$ by restricting the Schwartz kernel of $(\square_g - i \varepsilon)^{-\alpha}$ to the diagonal \cite{Dang2020,Dang2022}. The residues  are shown to be local geometric invariants analogous to the well-known Riemannian case, and they can be equivalently obtained through a generalization of the Guillemin--Wodzicki residue as a Pollicott--Ruelle  resonance \cite{Dang2021}.   

These parallels to the Riemannian case appear however to break down if one considers a Lorentzian \emph{Dirac operator} $\D$, the main reason being that the Hermitian form for which $\D$ is formally self-adjoint is \emph{no longer positive definite} in the Lorentzian setting. This has raised  uncertainties as to whether the program of non-commutative geometry centered on Dirac operators initiated by Connes and co-authors \cite{Connes1996,Connes2008a,Chamseddine1997,Chamseddine2007} can be carried out in Lorentzian signature, even though many attempts at generalizing  the formalism exist,  see e.g.~\cite{Moretti2003,VanSuijlekom2004,Paschke2004,Strohmaier2006,DAndrea2016,Devastato2018,martinetti} and references therein for results mostly focused on settings with special symmetries. In a different spirit,  a geometric index theory for Lorentzian Dirac was pioneered by Bär--Strohmaier \cite{Bar2019,Baer2020}, we note however that this does not address the question of existence of a spectral zeta function. 

In the present paper we consider the operator $P=-\D^2$ on small perturbations of Minkowski space and propose to view it as an operator which \emph{modulo a decaying term} is formally self-adjoint  for an auxiliary positive scalar product $\bra \cdot , \cdot\ket$ (for instance the one used in field quantization, see \eqref{eq:scal}), in the sense that $P- P^*\in \Psi_{\rm sc}^{1,-1}$  has coefficients decaying at  infinity (in the precise sense of the scattering calculus  $\Psi^{m,\ell}_\sc$ recalled in see \sec{ss:bi}) and has small $\Psi_{\rm sc}^{1,-1}$ seminorms.  We show that despite being an order $1$ perturbation of a formally self-adjoint operator, $P$ enjoys relatively good spectral properties.

\begin{theorem}[cf.~Theorems \ref{thm:sp}  and \ref{thm:sp2}] \label{thm1} Let  $\D\in {\rm Diff}^1_{\sc}(M;E)$ be a Dirac operator on a small perturbation of Minkowski space (see Definition \ref{def:SM}), and  let $P =-\D^2$ acting on test sections. Then the spectrum of the closure of $P$   consists at most of $\rr$ and isolated points that lie in  $\{ 0<\module{\Im \lambda} \leq R \}$ for some $R>0$. 
Moreover, the resolvent $(P-z)^{-1}:C^\infty_{\rm c}\to \pazocal{D}'$ is meromorphic in $\{\Im z \neq0\}$ with poles of finite multiplicities (called {resonances}). For each resonance $z_0$, any solution $u$ of $Pu=z_0u$ is  $C^\infty$ in $M^\inti$.
\end{theorem} 

At this point the precise structure of Dirac operators is not important and only the analytic features of $P$ matter  (this involves being a small perturbation of a non-trapping operator in the sense of sufficiently small $\Psi^{2,0}_\sc$ semi
norms), see Theorems \ref{thm:sp}  and \ref{thm:sp2} for  more general statements, with  sufficiently small perturbations of Minkowski space  being a special case  as discussed in the paragraph following   Definition \ref{def:SM}.  

Theorem \ref{thm1} allows us to define complex powers $(P-i \varepsilon)^{-\alpha}$ for $\varepsilon>0$ using the resolvent of $P$ through a complex contour integral that bypasses resonances (see Definition \ref{def:cp}). While the definition of $(P-i \varepsilon)^{-\alpha}$ may depend on the choice of integration contour, we show that the  resulting ambiguity is at most a finite rank smoothing operator.  We then combine the microlocal estimates used in the proof of  Theorem \ref{thm1} with Hadamard parametrix techniques to prove the following result on the spectral $\zeta$-function density of $P-i \varepsilon$,  the residues of which are not affected by the finite rank ambiguity. As in \cite{Dang2020}  we focus on the case when $n=\dim M$ is even.  

\begin{theorem}[cf.~Theorem \ref{thm:final}]  \label{thm2} For all $\varepsilon>0$, the Schwartz kernel of $(P-i \varepsilon)^{-\cv}$ has for $\Re\cv>\frac{n}{2}$ a well-defined on-diagonal restriction $(P-i \varepsilon)^{-\cv}(x,x)$, which extends as a meromorphic function with poles at $\{\n2,$ $\n2-1$, $\n2-2$, $\dots$, ${1}\}$. Furthermore, 
$$ 
\bea 
\lim_{\varepsilon\rightarrow 0^+} \res_{\cv=\frac{n}{2}-1} \tr_E\left((P- i\varepsilon)^{-\cv}\right)(x,x)= \frac{ \rk (E) R_g(x)  }{{i}6(4\pi)^{\n2} \Gamma\big(\frac{n}{2}-1\big) } +  \frac{  2 \tr_E \big( F^E\big)(x)  }{{i}(4\pi)^{\n2} \Gamma\big(\frac{n}{2}-1\big) },
\eea 
$$
where $R_g$ is the scalar curvature of $(M,g)$ and $F^E$ is the twisting curvature of $E\to M$. 
\end{theorem}

In particular this gives a positive answer to the question of the validity of Connes' spectral action principle in Lorentzian signature.  We note that one can  compute other residues, and it is also possible to obtain a small $h$ expansion for $f(h(P+i\varepsilon))$ for suitable Schwartz functions $f$ similarly as in \cite{Dang2020} and recover the Lorentzian analogues of heat kernel coefficients, see Remark \ref{smallh}.

\subsection{Idea of proofs, remarks} In proving Theorem \ref{thm1}, the primary difficulty is that existing results for  $\square_g$ based on the scattering calculus $\Psi^{m,\ell}_\sc$ rely on formal self-adjointness in crucial steps of the proofs. In particular, while the microlocal propagation and radial estimates employed in \cite{vasyessential} generalize well, they need to be supplemented by an extra borderline positive commutator argument which does not. In a nutshell, the reason why the former  are insufficient is that  they provide regularity and decay statements \emph{conditionally to threshold conditions} at the radial sets; this is in fact characteristic of radial estimates as pioneered by Melrose \cite{melrosered}. In practice it  means here  that  for $\lambda \in \cc \setminus \rr$ one can obtain the implication
\beq\label{eq:impli}
u\in L^2, \  (P-\lambda)u= 0  \implies  u \in H_{\rm sc}^{\12,-\12-\epsilon}
\eeq
for $\epsilon>0$ (here $H_{\rm sc}^{s,\ell}$ are scattering Sobolev spaces recalled in \sec{ss:bi}), but beyond that extra arguments are needed (we note that as remarked in \cite{Taira2020a},  under more restrictive non-trapping assumptions \eqref{eq:impli} follows already from a local smoothing estimate of Chihara \cite{Chihara2002}).

 To tackle this, the idea is to replace the $\Psi_{\sc}^{m,\ell}$ calculus, which microlocalizes the fiber-compactified scattering  cotangent bundle $\co$, by a more precise resolved calculus $\Psiscq^{m,k,\ell}$ obtained by blowing up the corner  $\corner$  of $\co$.   This resolution process produces a calculus which microlocally in the front face interior is equivalent to the \emph{quadratic scattering calculus}  \cite{Wunsch1999a}, see Remark \ref{rem:qsc}. It is analogous to a construction of second microlocalization, in
 which the well-behaved perspective is blowing up the corner
of the fiber-compactified $\b$-cotangent bundle, see
 \cite{vasylap,vasyresolvent}.  As a result we get more flexibility on the level of threshold conditions and results in more precise Fredholm estimates with greater room for trade-off between regularity and decay in the relevant microlocal regions; one can think of this as a fully microlocal point of view which in particular incorporates the subelliptic estimate of Taira \cite{Taira2020a}. 
 
 At this point we remark that in parallel to the present work, Jia--Molodyk--Sussman \cite{JMS} show that $\square_g$ is essentially self-adjoint on more general asymptotically Minkowski, allowing also for perturbations of Minkowski in the sense of Einstein metrics. The result uses radial estimates in the $\Psi_{\rm de,sc}^{m,\ell}$ calculus recently developed by Sussman \cite{Sussman2023a} on the blow-up of  null infinity---these turn out to be robust enough to bypass the problematic thresholds and provide an alternative to our $\Psiscq^{m,k,\ell}$ calculus, advantageous in terms of allowing for weaker assumptions on the spacetime metric. 
 
The second issue in the non self-adjoint case is proving the invertibility of $P-\lambda$  for  large $\Im \lambda$, needed to conclude Theorem \ref{thm1} by analytic Fredholm theory. One way pursued in this paper is to show large $\Im \lambda$ versions of the  $\Psiscq^{m,k,\ell}$  based estimates through  adapting the underlying positive commutator arguments, with decay in $\Im \lambda$ traded for regularity. This has however the disadvantage of not accounting for large  $\Im \lambda$  elliptic estimates: these are needed to complete the Fredholm estimates in the resolved setting,  but they require extra arguments. A standard approach would be to replace by $\lambda$ by $\lambda/h^2$, and derive semi-classical versions of the estimate (see Vasy--Zworski \cite{Vasy2000} for the asymptotically Euclidean case using the $\Psi_\sc^{m,\ell}$ calculus), with the requirement    that we need small $h$ estimates on  \emph{non semi-classical} spaces: in fact to prove Theorem \ref{thm2} we need to show microlocal mapping properties of the resolvent that survive integration along an infinite complex contour. To address this systematically we introduce a blow-up of $h=0$ at fiber infinity in the spirit of the recent work of Vasy \cite{Vasy:Semiclassical-standard} on the relationship between the semiclassical and standard pseudodifferential algebras. Namely, we construct a mixed non semiclassical-semiclassical calculus $\Psi_{\qscl,\scl,\semi^2,\semi,\mathrm{cl}}^{m,k,\ell,p,q,r}$ by  introducing the semiclassical phase space into the classical
 one by blowing up fiber infinity in the parameter-dependent fiber compactified scattering cotangent bundle at $h=0$. This gives a fully microlocal framework using which we derive new Fredholm estimates  which have  \emph{all the desired features simultaneously}: they imply in particular estimates with  adjustable trade-off between behavior in $h$ and regularity/decay, as follows from the relationships $\Psi_{\qscl,\scl,\semi^2,\semi,\mathrm{cl}}^{m,k,\ell,p,q,r}$-based spaces with more standard weighted Sobolev spaces; this suffices to prove Theorem \ref{thm1}.
 
Theorem \ref{thm2} is then concluded largely following the Hadamard parametrix-based arguments in \cite{Dang2020} with necessary adjustments for vector bundles, provided that we have a sufficiently precise bound on the wavefront set of $(P-\lambda)^{-1}$ uniformly in $\lambda$, with decay along the integration contour. Here we obtain it directly from the radial and propagation estimates used to prove Theorem \ref{thm1}, thus bypassing the evolutionary parametrix construction in \cite{Dang2020} and removing the explicit global hyperbolicity assumption made there.
 
 \subsection{Plan of the paper} The paper is organized as follows. In \sec{s:propagation} we introduce the resolved $\Psiscq^{m,k,\ell}$ calculus and prove propagation of singularities estimates (Proposition \ref{pos2}) and radial estimates (Propositions \ref{radial3} and \ref{radial4}) in the associated resolved Sobolev spaces. These are then used in \sec{s:resolvent} to derive Fredholm estimates and microlocal resolvent estimates. Semi-classical arguments needed to conclude the large $\Im \lambda$ versions of the estimates are given in \sec{s:semicl} which introduces the resolved $\Psi_{\qscl,\scl,\semi^2,\semi,\mathrm{cl}}^{m,k,\ell,p,q,r}$ calculus estimates and discusses estimates in this setting  and relationships with non-semiclassical spaces. Finally, \sec{s:dirac} specializes to the case $P= - \D^2$ and proves the relationship between complex powers and local invariants.

 \section{Propagation estimates in resolved setting}\label{s:propagation}\init
 

\subsection{Notation}\label{ss:notation} Let $M$ be a compact manifold with boundary, and let $g$ be a pseudo-Riemannian metric on its interior $M^\inti$. We denote by $L^2(M)$ the canonical $L^2$ space associated to the volume density $d\vol_g$ of $g$, i.e.~the squared $L^2(M)$ norm  is
\beq\label{eq:defL2}
\| u \|^2= \int_{M^\inti} \left| u(x) \right|^2d\vol_g.
\eeq

We use the notation $\cf(\M)$ for the space of smooth function on $\M$, meant in the usual sense of being smoothly extendible across the boundary $\p \M$. More generally, if $E\xrightarrow{\pi} M$ is a smooth vector bundle then $\cf(M;E)$ stands for the space of its smooth  sections. If $E$ is a Hermitian bundle (in the sense that the hermitian form on each fiber  is a scalar product, so in particular it is assumed to be positive definite), we denote by $L^2(M;E)$ the corresponding Hilbert space of  square-integrable sections, with norm given by the analogue of \eqref{eq:defL2} with $\module{\cdot}$ replaced by the fiberwise norm $\module{\cdot}_E$.
 
 Let $\bdf$ be a boundary-defining function of $\p \M$, i.e.~a function $\bdf\in\cf(\M)$ such that 
$\bdf>0$ on $M$, $\p \M=\{\bdf=0\}$, and $d\bdf\neq 0$ on $\p \M$. By the collar neighborhood theorem, there exists $W\supseteq\pM$, $\epsilon>0$ and a diffeomorphism $\phi:\clopen{0,\epsilon}\times \pM \to W$ such that $\bdf\circ \phi$ is the projection to the first component of $\clopen{0,\epsilon}\times \pM$. We  use notation proper of $\clopen{0,\epsilon}\times \pM$ and disregard $\phi$ in the notation when working close to the boundary, i.e.~in the collar neighborhood $W$.  In this sense, we can use  local coordinates of the form $(\bdf,y_1,\dots,y_{n-1})$, where  $(y_1,\dots,y_{n-1})$ are local coordinates on the boundary.

\subsection{Pseudo-Riemannian \texorpdfstring{$\sc$}{sc}-spaces}\label{ss:scmetrics} 
Let us recall that $\cV_\b(M)$ is the space of \emph{$\b$-vector fields} on $M$, i.e.~the space of vector fields in $\cf(M;TM)$ which are tangent to the boundary $\pM$.  The same definition applies more to the more general case when $M$ is replaced by a manifold with corners ($\b$-vector fields are then tangent to all boundary hypersurfaces); this will become useful when considering vector fields on compactified versions of phase space. The space of \emph{$\sc$-vector fields} is by definition $\cV_\sc(M)=\bdf \cV_\b(M)$.
More explicitly,  each $V\in \cf(\M;\be T\M)$ is  locally of the form
\beq\label{sctvf}
V=V_0(\bdf,y) \bdf^2\p_\bdf + \sum_{i=1}^{n-1} V_i(\bdf,y)\bdf \p_{y_j}, \ V_0,V_i \in \cf({U}), \ i=1,\dots,n-1
\eeq
on {a} chart neighborhood ${U}$ with local coordinates  $(\bdf,y_1,\dots,y_{n-1})$ as above.  

In Melrose's $\sc$-geometry \cite{melrosered} (or scattering geometry) the ${\rm sc}${-tangent bundle}  $\Tsc M$ is the  unique vector bundle over $\M$ such that $\cV_\sc(M)=\cf(M;\Tsc M)$.  The ${\rm sc}${-cotangent bundle} $\be T^*\M$ is  the dual bundle of $\be T\M$. Thus, in local coordinates $(\bdf,y_1,\dots,y_{n-1})$, the smooth sections of $\be T^*\M$ are $\cf(\M)$-generated by $(\bdf^{-2}d\bdf, \bdf^{-1} d y_1, \dots, \bdf^{-1} d y_{n-1})$, in the same way that smooth sections of $\be T\M$ are $\cf(\M)$-generated by $(\bdf^2 \p_\bdf, \bdf \p_{y_1}$, $\dots, \bdf \p_{y_{n-1}})$. 

\begin{definition} An \emph{$\sc$-metric} is a  non-degenerate smooth section of the fiberwise symmetrized tensor product bundle $\be T^*\M \otimes_{\rm s} \be T^*\M$. 
\end{definition}

Correspondingly,  $(\M,{g})$ is a \emph{pseudo-Riemannian $\sc$-space of signature $(k,n-k)$}   if  ${g}$ is a pseudo-Riemannian metric on $M^\inti$ of signature $(k,n-k)$,  which extends to an $\sc$-metric. Here, the convention is that $k$ denotes the number of ``pluses'', and  $(\M,{g})$ is a \emph{Lorentzian $\rm sc$-space} if  $k=1$.

In what follows we assume that $(\M,{g})$ is a pseudo-Riemannian $\rm sc$-space. The volume density of $(M,g)$, $d\vol_g$, extends then to an \emph{${\rm sc}$-density} on $\M$, i.e.~in local coordinates $(\bdf,y)$ it is of the form $\nu(\bdf,y)\left|\bdf^{-2} d\bdf\, \bdf^{-n+1}d y  \right|$ for some $\nu\in \cf(\M)$. For the sake of generalities on pseudo-differential operators discussed in this chapter  one could actually take any smooth manifold $M$ with boundary and equip it with an ${\rm sc}$-density to define $L^2(M)$.

\begin{example}\label{ex:mi} The standard example is  $M^\inti=\rr^n$, with $\M=\overline{\rr^n}$  the \emph{radial compactification of $\rr^n$}. Let us recall that $\overline{\rr^n}$  is by definition  the quotient of $\rr^n\sqcup\big(\clopen{0,1}_\bdf \times \sphere^{n-1}_y \big)$ by the  relation which identifies any  $x\in\rr^n\setminus\{0\}$ with the point $(\bdf,y)$, where $\bdf=r^{-1}$ and $(r,y)$ are the polar coordinates of $x$. The smooth structure near $\{\bdf=0\}$ is the natural one in $(\bdf,y)$ coordinates,  and then $\overline{\rr^n}$ is diffeomorphic to a closed ball.  The standard frame $(\p_{x_0},\dots,\p_{x_{n-1}})$ on $T\rr^n$ smoothly extends to $\be T^*\overline{\rr^n}$ (this is easily seen by writing $\p_{x_j}$ in terms of the $\sc$-vector fields $\p_r=-\bdf^2\p_\bdf$ and $\p_{y_{j}}$). Conversely, any $V\in \cV_{\rm sc}(M)$ is in the $\cf(\overline{\rr^n})$-span of $(\p_{x_0},\dots,\p_{x_{n-1}})$, i.e.~the coefficients  extend smoothly across $\{\bdf=0\}$ in $(\bdf,y)$ coordinates  on top of being smooth in $\rr^n$. Similarly, the flat metric $g^0=dx_0^2+ \cdots + dx_{k-1}^2- (dx_{k}^2+\cdots+dx_{n}^2)$ on $\rr^n$ extends to an $\rm sc$-metric on $\overline{\rr^n}$.
\end{example}

\subsection{The \texorpdfstring{$\rm sc$}{sc}-calculus}\label{ss:bi}

Following the approach in \cite{vasygrenoble,vasyessential,hassell}, the example of $\overline{\rr^n}$ plays for us the role of a reference model, suitable for formulating various definitions. Then in a second step,      these definitions are transplanted to the general case using the identification of a coordinate neighborhood of $\pM$ with a coordinate neighborhood of a point of
$\p\overline{\rr^n}$ (in analogy to how standard notions on manifolds without boundary are defined by identifying coordinate charts with open subsets of $\rr^n$). To that end one can use for instance  diffeomorphisms of the form
\beq\label{eq:diffeo}
M^\inti\supset U \ni (\bdf,y) \mapsto   \bdf^{-1} \psi(y)\in \rr^n,
\eeq
where $\psi$ is a diffeomorphism from a small open set of $\p M$ to an open set of the unit sphere  $\{ r=1\}\subset \rr^n$.

In particular, this gives a convenient way of introducing the class of \emph{scattering
pseudodifferential operators} $\Psisc^{m,\ell}(M)$, $m,\ell\in \rr$, by reduction to the $\overline{\RR^n}$ case, or equivalently to an appropriate uniform
structure in the interior. 

Namely, in the model case, using standard coordinates $(x,\xi)$ on $T^*\rr^n=\rr^n \times \rr^n$ and setting $\bra x \ket=(1+|x|^2)^\12$ and $\bra \xi \ket=(1+|\xi|^2)^\12$,  the class of \emph{$\rm sc$-symbols} (product-type symbols) of order $m,\ell$, denoted by $S^{m,\ell}(\Tsc^*\overline{\rr^n})$,  is the set of all $a\in\cf(T^*\rr^n)$ such that
  \beq\label{eq:estisymb}
  \forall\alpha,\beta\in\nn^{n}, \ \   \big| \p_x^\alpha   \p_\xi^\beta a(x,\xi)    \big|\leqslant  C_{\alpha\beta}   \bra x\ket^{\ell-|\alpha|}   \bra \xi\ket^{m-|\beta|}. 
 \eeq
 For general $\M$, the symbol class $S^{m,\ell}({\Tsc^*}M)$ is then defined by reduction to $\overline{\RR^n}$. Note that symbols in $S^{m,\ell}({\Tsc^*}M)$ are in general not smooth sections of ${\Tsc^*}M$, but can be seen as sections satisfying a weaker conormality property at infinity, see \eqref{eq:estisymb2} and the neighboring paragraph.

Next, the class  of \emph{$\sc$-pseudo-differential operators} $\Psi^{m,\ell}_{\rm sc}(M)$ is by definition
\beq\label{eq:defpsi}
\Psi^{m,\ell}_{\rm sc}(M) \defeq \Op \big( S^{m,\ell}(\Tsc^*M) \big) + \cW^{-\infty,-\infty}_{\rm sc}(M),
\eeq
where $\Op$ is a quantization map (defined using a partition of unity $\{\chi_i\}_i$ subordinate to a finite chart covering of $M$, and with  chart diffeomorphisms of the form  \eqref{eq:diffeo}  close to $\p\M$), and  $\cW^{-\infty,-\infty}_{\rm sc}(M)$ is a class of regularizing operators in the sense that their Schwartz kernels are smooth  and decrease rapidly (with all derivatives) at large separations. We refer the reader to, e.g., \cite{vasygrenoble} and \cite[\S 2]{Uhlmann2016} for an introduction, cf.~\cite{melrosered} for the original, more geometric description of the Schwartz kernels of scattering pseudo-differential operators.

 When discussing microlocalisation it is useful to compactify the fibers of $\be T^*\M$. The base manifold $\M$ having  a boundary already, the {fiberwise radial compactification} of $\be T^*\M$ yields a manifold with corners {(see \cite[\S6.4]{melrosered} for details)} denoted by $\co$. The structure of $\co$ as a manifold with corners  can be deduced from the model case $\overline{\rr^n}$, where
  $$
 {\overline{\be T^*\rr^n}}=\overline{\RR^n_x}\times\overline{\RR^n_\xi}.
  $$
is obtained by  radially compactifying the position and the momentum space (separately). The two boundary hypersurfaces are $\overline{\RR^n}\times
  \p \overline{\RR^n}$, called {\em fiber infinity}, and $\p\overline{\RR^n}\times
  \overline{\RR^n}$, called {\em base infinity}, and they intersect at the
  corner  $\p\overline{\RR^n}\times
  \p\overline{\RR^n}$.  In the general case  the manifold with corners $\co$ has again two boundary hypersurfaces:   fiber infinity is the boundary of the fiber compactification, denoted by $\fibinf$ (it is in fact the \emph{$\sc$-cosphere bundle} of $M$), and base infinity is $\basinf$.  We can consider $\bdf$ as boundary defining function of base infinity $\basinf$. As boundary defining  function of fiber infinity $\fibinf$ we  take $$
 \rho_\infty=\langle(\tau,\mu)\rangle^{-1}=(|(\tau,\mu)|^2+1)^{-1/2},
 $$ where $\module{\cdot}$ is the length with
 respect to an arbitrarily chosen Riemannian $\rm sc$-metric and $(\tau,\mu_1,\ldots,\mu_{n-1})$  are coordinates on the fibers of $\be T^*\M$, sc-dual to 
 $(\bdf,y_1,\ldots,y_{n-1})$, i.e.~sc-covectors are written as
 \beq\label{eq:defmu}
 \tau\,\frac{d\bdf}{\bdf^2}+\sum_{j=1}^{n-1}\mu_j\,\frac{dy_j}{\bdf}\defeq \tau\,\frac{d\bdf}{\bdf^2}+\mu \cdot \frac{dy}{\bdf}.
 \eeq
  The locus of microlocalisation is the boundary of $\co$,
  $$
  \p\co=\basinf\cup\fibinf,
  $$
 including the corner $\corner$.  From this point of view, the  symbol class $S^{m,\ell}({\Tsc^*}M)$ can be equivalently defined   as the set of all smooth sections $a$ of $T^*M^\inti$
 such that for all $N\in \nn$ and all $V_j\in \cV_{\b}(\co)$, $1\leq j \leq N$,
 \beq\label{eq:estisymb2}
  \rho^\ell  \rho^{m}_\infty V_1 \dots  V_N  a \in L^\infty(T^* M^\inti).
\eeq
In fact, the equivalence with the original definition  \eqref{eq:estisymb} is easily obtained by  writing  the $\b$-vector fields in coordinates  and appropriately  commuting derivatives and weights.  

The subclass of \emph{classical symbols} $S^{m,\ell}_{\rm cl}({\Tsc^*}M)$ (characterized by having a joint one-step polyhomogeneous expansion  in powers of $\bra x \ket$ and $\bra \xi \ket$ at infinity) can be invariantly defined as  $S^{m,\ell}_{\rm cl}({\Tsc^*}M)=\rho^{-\ell} \rho^{-m}_{\infty} \cf(\co)$. Note that this corresponds to requiring \eqref{eq:estisymb2} for \emph{all} smooth vector fields $V_j$ rather than merely for  those which are tangent to the boundary hypersurfaces  (in other words,   conormality  at $\p\co$ is replaced by the stronger property of smooth extendibility).  
 
 In the context of propagation estimates it is useful to introduce more general symbol classes with \emph{varying decay order}  $\ell\in S^{0,0}(\Tsc^*M)$. For this
 purpose one also needs to relax the type of  symbol estimates
 and allow for small power losses, which can be used to absorb logarithmic losses arising when taking derivatives and hitting $\ell$. Namely, for $\delta\in \open{0,\12}$  one introduces the class $S^{m,\ell}_{\delta}(\Tsc^*M)$ by reduction to symbols on $\rr^n\times \rr^n$ satisfying  the bound
 $$
   \forall\alpha,\beta\in\nn^{n}, \ \ \big|\p_x^\alpha \p_\xi^\beta a(x,\xi)\big|\leq C_{\alpha\beta}\langle
 x\rangle^{\ell(x,\xi)-|\alpha|+\delta(|\alpha|+|\beta|)}\langle \xi\rangle^{m-|\beta|+\delta(|\alpha|+|\beta|)}
 $$
 instead of \eqref{eq:estisymb}.  
The corresponding pseudo-differential class will still be denoted by $\Psi^{m,\ell}_\sc(M)$.  The most significant consequence of the presence of $\delta$ is that  for $A\in\Psisc^{m,\ell}(M)$ and $B\in\Psisc^{m',\ell'}(M)$, 
 $$
 [A,B]\in\Psisc^{m+m'-1+2\delta,\ell+\ell'-1+2\delta}(M)
 $$
 instead of being in $\Psisc^{m+m'-1,\ell+\ell'-1}(M)$. 
However, in our applications we can take $\delta$ arbitrarily close to $0$, so it will have little practical significance and  we will frequently disregard it in the notation.
 
  In the $\sc$-calculus, the \emph{principal symbol}  of $A\in\Psisc^{m,\ell}(M)$ is the equivalence class of the symbol of $A$  in $S^{m,\ell}(\Tsc^*M)/ S^{m-1,\ell-1}(\Tsc^*M)$, or in $$S^{m,\ell}_{\delta}(\Tsc^*M)/ S^{m-1+2\delta,\ell-1+2\delta}_{\delta}(\Tsc^*M)$$ if we make the dependence on $\delta$ explicit.  In the simplest case of $A\in \Psisc^{0,0}(M)$ classical, it is possible to identify the principal symbol with the restriction of $a\in \cf(\co)$ to $\p \co$. For  classical  $A\in \Psisc^{m,\ell}(M)$ of arbitrary order  there is also a natural identification of the principal symbol with a function on $\p\co$.
  
  
  The \emph{microsupport} $\wf'_\sc(A)$ of $A\in \Psi^{s,\ell}_\sc(M)$ is the complement of the set of points $q\in \p\co$ such that the (full) symbol of $A$ coincides in a neighborhood of $A$ with a symbol in $S^{-N,-L}(\Tsc^*M)$ for all $N,L\in\rr$. If $A$ is classical, then its \emph{elliptic set} is the  complement $\elll_\sc(A)=\p\co\setminus \Char_\sc(A)$ of the \emph{characteristic set} $\Char_\sc(A)$, defined as the closure of the zero set of the principal symbol.

 Now, for $m\geqslant 0$ and  $\ell\in \cf(\co )$, one can define the \emph{weighted Sobolev space of variable weight order} as follows:
  $$
  H^{m,\ell}_\sc(M)= \{ u \in L^2(M) \ | \ A u \in L^2(M)\},
  $$
  where $A\in\Psi^{m,\ell}_\sc(M)$ is a classical elliptic operator (i.e., $\elll_\sc(A)=\p\co$) which can be chosen arbitrarily. One can fix in particular an invertible $A$, and  the norm can be then defined as  $\| u\|_{m,\ell}= \|A u \|$ (different choices of $A$ give equivalent norms). For $m\leqslant  0$, $H^{m,\ell}_\sc(M)$ can be defined as the dual of $H^{-m,-\ell}_\sc(M)$. Note that with these conventions for $  H^{m,\ell}_\sc(M)$, higher $m$ means more regularity, and higher $\ell$ means more decay at $\p \M$. 
  Correspondingly, the infinite order Sobolev spaces
  $$
  H^{-\infty,-\infty}_\sc(M)\defeq  \textstyle\bigcup_{m,\ell} H^{m,\ell}_\sc(M), \quad   H^{\infty,\infty}_\sc(M)\defeq  \textstyle\bigcap_{m,\ell} H^{m,\ell}_\sc(M),
  $$
  are the natural generalizations of the space of tempered distributions on $\rr^n$, resp.~the space of Schwartz functions on $\rr^n$. In terms of the former,  for all $m,\ell\in\rr$ and for any elliptic $A\in\Psi^{m,\ell}_\sc(M)$ we simply have 
  $$
   H^{m,\ell}_\sc(M)= \big\{ u \in  H^{-\infty,-\infty}_\sc(M)  \ \big|\big. \ A u \in L^2(M)\big\}.
  $$

In the calculus it is important that we can compute principal symbols of commutators. If  $A\in\Psisc^{m,\ell}(M)$ and $B\in\Psisc^{m',\ell'}(M)$ with $a$, resp.~$b$, denoting the principal symbols of $A$,
resp.~$B$, the principal symbol of $i [A,B]$ is the Poisson bracket
$\{a,b\}$ obtained from the symplectic structure on  $T^*M^\inti$. In the ${\rr^n}$ case, it is  given by the familiar expression
$$
\{a,b\}=\sum_{j=1}^n
(\p_{x_j}a)(\p_{\xi_j}b)-(\p_{\xi_j} a)(\p_{x_j}b) \defeq H_a b,
$$
which defines the Hamilton vector field $H_a$. In coordinates $(\tau,\mu_1,\dots,\mu_{n-1})$ introduced above \eqref{eq:defmu}, a straightforward change of variables computation  gives
\begin{equation}\bea\label{eq:Ham-vf}
H_a =\bdf \big((\p_\tau
a)(\bdf\p_\bdf+\mu\cdot\p_\mu)-\big((\bdf\p_\bdf+\mu\cdot\p_\mu)a\big)\p_\tau
+(\p_{\mu}a)\cdot\p_{y}-(\p_{y}a)\cdot\p_{\mu}\big).
\eea\end{equation}
In view of the homogeneity in $\bdf$ and $\rho_\infty$, it is natural to introduce the rescaled Hamilton vector field
$$
\xoverline{H}_a\defeq  \bdf^{\ell-1} \rho_\infty^{m-1}H_a.
$$
 If $a$ is classical,  this extends to a smooth vector field  $\xoverline{H}_a\in \cV_\b(\co)$, and
 thus defines a flow on $\co$ (the {\em Hamilton flow}). For general $a\in S^{m,\ell}(\Tsc^*M)$, $\xoverline{H}_a$ has merely conormal coefficients (of order $(0,0)$) as a vector field tangent to the boundary $\p\co$.
 
\emph{Radial points} are  points of $\p\co$ at which  $\xoverline{H}_a$ vanishes {as
 a vector in $T\co$}. Note that this does not necessarily mean that $\xoverline{H}_a$ vanishes there as a
 $\b$-vector field; in particular, non-zero $\cf(\co)$-linear combinations of $\bdf\p_\bdf$  and $\rho_\infty \p_{\rho_\infty}$ are allowed.
 
 \subsection{Calculus in resolved setting} An important aspect of the work here is a resolution of the compactified scattering
 phase space $\co$.  In fact, in the context of Fredholm estimates for operators of the form $P-\lambda$ with $P\in \Psi^{2,0}_\sc(M)$ and $\lambda\in \cc$, it will turn out that  this resolution properly
 reflects the transition of the role of $\Im \lambda$ from being
 principal at $\basinf$, to being
 sub-subprincipal at $\fibinf$. 
 
 Namely, we consider the blow-up of the corner $\corner$, in order to obtain the resolved
 phase space
 $$
  \cores=[\co; \corner],
 $$ 
 equipped with a blow-down map $\beta: [\co; \corner]\to\corner$. Thus, the corner $\corner$ is replaced by a boundary hypersurface of    $\cores$ called the \emph{front face} and henceforth denoted by $\ff$. Fiber infinity and base infinity  lift to two boundary hypersurfaces denoted respectively by $\fibi$ and $\basi$ (see Figure \ref{fig:blowup}).

  \begin{figure}
  \tikzset{every picture/.style={line width=0.75pt}} 

\tikzset{every picture/.style={line width=0.75pt}} 

\begin{tikzpicture}[x=0.75pt,y=0.75pt,yscale=-0.7,xscale=0.7]

\draw [line width=0.75]    (80,66.67) -- (176.75,66.67) ;
\draw  [draw opacity=0][line width=0.75]  (264.7,153.33) .. controls (264.7,153.33) and (264.7,153.33) .. (264.7,153.33) .. controls (216.12,153.33) and (176.75,114.53) .. (176.75,66.67) -- (264.7,66.67) -- cycle ; \draw  [line width=0.75]  (264.7,153.33) .. controls (264.7,153.33) and (264.7,153.33) .. (264.7,153.33) .. controls (216.12,153.33) and (176.75,114.53) .. (176.75,66.67) ;  
\draw [line width=0.75]    (264.7,153.33) -- (264.7,240) ;
\draw    (307.06,141) -- (378,141) ;
\draw [shift={(380,141)}, rotate = 180] [fill={rgb, 255:red, 0; green, 0; blue, 0 }  ][line width=0.08]  [draw opacity=0] (12,-3) -- (0,0) -- (12,3) -- cycle    ;
\draw [line width=0.75]    (410,80) -- (571,80) ;
\draw [line width=0.75]    (571,80) -- (570,240) ;

\draw (116.91,48.44) node [anchor=north west][inner sep=0.75pt]    {$\fibi$};
\draw (200.32,111.07) node [anchor=north west][inner sep=0.75pt]    {$\ff$};
\draw (263.53,184.25) node [anchor=north west][inner sep=0.75pt]    {$\basi$};
\draw (459,61.4) node [anchor=north west][inner sep=0.75pt]    {$\fibinf$};
\draw (331,122.4) node [anchor=north west][inner sep=0.75pt]    {$\beta$};
\draw (572,152.4) node [anchor=north west][inner sep=0.75pt]    {$\basinf$};
\draw (556,62.4) node [anchor=north west][inner sep=0.75pt]    {$\corner$};

\end{tikzpicture}
  \caption{\label{fig:blowup} Blow-up $\,\cores=[\co; \corner]$ of the corner $\corner$ of $\co$. The $\fibi$-face is the lift of fiber infinity $\fibinf=\{ \rho_\infty = 0\}$, and the $\basi$-face is the lift of base infinity $\basinf=\{\bdf =0 \}$.}
  \end{figure}
 
 As an illustration, in the $\RR^n$-setting the above definition means a blow up of the corner
 $\p\overline{\RR^n}\times\p\overline{\RR^n}$, which yields the resolved space
 $
 [\overline{\RR^n}\times\overline{\RR^n}; \p\overline{\RR^n}\times\p\overline{\RR^n}].
 $
 Near the corner, $\overline{\RR^n}\times\overline{\RR^n}$ has the
 structure
 $$
 \clopen{0,1}_\rho\times \clopen{0,1}_{\rho_\infty}\times
 \p\overline{\RR^n}\times\p\overline{\RR^n}.
 $$
 The last two factors
 are unaffected by the blowup. On the other hand, a neighborhood of the
 corner in the quadrant
 $ \clopen{0,1}_\rho\times \clopen{0,1}_{\rho_\infty}$ is replaced by its polar coordinate
 version, $\clopen{0,1}_R\times \clopen{0,\pi/2}_\theta$.  In practice, it is simpler
 to work with projective coordinates and use two charts. Thus, where
 $C\rho>\rho_\infty$, i.e.\ near the lift of fiber infinity, where $x$ is
 relatively large, one has coordinates
 $$
 (\rho,\rho_\infty/\rho)\in\clopen{0,1}\times\clopen{0,C},
 $$
 while in $C\rho_\infty>\rho$,
 i.e.\ where $\rho_\infty$ is relatively large, which is near the lift
 of the boundary fibers of the cotangent bundle, one has coordinates
 $$
 (\rho_\infty,\rho/\rho_\infty)\in\clopen{0,1}\times\clopen{0,C}.
 $$
 
 Moving on to the general case,  examples of boundary defining functions of respectively $\fibi$, $\ff$, $\basi$,  are respectively:
 $$
 \bea
\rho_\fibi&:=(1+\rho/\rho_\infty)^{-1} =     \frac{\rho_\infty}{\rho+\rho_\infty}, \\  \rho_\ff&:=\rho+\rho_\infty, \\   \rho_\basi&:=(1+\rho_\infty/\rho)^{-1} = \frac{\rho}{\rho+\rho_\infty}.
\eea
 $$

 \begin{definition} For $m,k,\ell\in \rr$, the resolved symbol class $S^{m,k,\ell}({\Tsc^*}M)$ is the set of all smooth sections $a$ of $T^*M^\inti$
  such that for all $N\in \nn$ and all $V_j\in \cV_{\b}(\cores)$, $1\leq j \leq N$,
   \beq\label{eq:estisymb3}
   \rho_\fibi^{m}\, \rho_\ff^k\, \rho_\basi^\ell\, V_1 \dots  V_N  a \in L^\infty(T^* M^\inti).
  \eeq
 \end{definition}
 
Note that this is equivalent to saying that we require \eqref{eq:estisymb3} for all $V_j\in \cV_{\b}(\co)$.

 In analogy to the scattering calculus, \emph{classical symbols} in $S^{m,k,\ell}({\Tsc^*}M)$ are by definition elements of $\rho_{\fibi}^{-m} \rho_\ff^{-k} \rho_\basi^{-\ell} \cf(\cores)$.

By quantizing symbols in $S^{m,k,\ell}({\Tsc^*}M)$ we obtain a new pseudo-differential $*$-algebra.

 \begin{definition} For $m,k,\ell\in \rr$, the resolved scattering pseudo-differential class $\Psiscq^{m,k,\ell}(M)$ is obtained by quantizing elements of $S^{m,k,\ell}({\Tsc^*}M)$, i.e.
 \beq\label{eq:defpsi2}
\Psiscq^{m,k,\ell}(M) \defeq \Op \big( S^{m,k,\ell}(\Tsc^* M) \big) + \cW^{-\infty,-\infty}_{\rm sc}(M),
 \eeq
 where $\cW^{-\infty,-\infty}_{\rm sc}(M)$ is the same ideal  of smoothing operators as in  \eqref{eq:defpsi}.
 \end{definition}

 \begin{remark}\label{rem:qsc} The label ``$\scq$'' stands for ``scattering, quadratic scattering'', which is a way of indicating that microlocally in the interior of the new front face $\ff$, operators in $\Psiscq^{m,k,\ell}(M)$ are \emph{quadratic scattering pseudo-differential operators} as introduced by Wunsch \cite{Wunsch1999a} (this also is the same as \emph{isotropic
  operators}). This can be seen by observing that in the region of
 $\overline{\Tsc^*}M$ where say $|\tau|>c|\mu|$, $c>0$, coordinates
 on $\overline{\Tsc^*}M$ are given by 
 $x,y,|\tau|^{-1},\mu/|\tau|$, hence the front face projective variable
 (where $x$ is relatively large, i.e.\ near the lift of fiber infinity)
 becomes
 $|\tau|^{-1}/x=1/x|\tau|$.
 \end{remark}
 
The first crucial observation is that all operators in $\Psisc^{m,\ell}(M)$ are also pseudo-differential operators in the resolved class  $\Psiscq^{m,k,\ell}(M)$ for all $k\geqslant m+\ell$.  In fact, we can write
 \beq\label{eq:rororo}
   \rho_\fibi^{m}\, \rho_\ff^k\, \rho_\basi^\ell= (1+\rho/\rho_\infty)^{-m} (\rho+\rho_\infty)^k  (1+\rho_\infty/\rho)^{-\ell}= \rho^m_\infty \rho^\ell (\rho+\rho_\infty)^{k-m-\ell}.
 \eeq
 Furthermore, $\b$-vector fields on $\co$ lift to $\b$-vector fields on $\cores$. In consequence, by comparing  the definitions  \eqref{eq:estisymb2} and \eqref{eq:estisymb3} of the respective symbol classes we get immediately the inclusion
 $$
S^{m,\ell}({\Tsc^*}M) \subset  S^{m,m+\ell,\ell}({\Tsc^*}M),
 $$
 and consequently $\Psisc^{m,\ell}(M) \subset \Psiscq^{m,m+\ell,\ell}(M)$ on the  level of operators.  In the opposite direction,   from \eqref{eq:rororo} we can see that away from the $\fibi$-face (i.e., where $\rho_\infty$ is relatively large), $S^{m,k,\ell}({\Tsc^*}M)$ coincides with $S^{k-\ell,\ell}({\Tsc^*}M)$. Similarly, away from $\basi$ (i.e., where $x$ is relatively large), $S^{m,k,\ell}({\Tsc^*}M)$ coincides with $S^{m,k-m}({\Tsc^*}M)$. In other words, $\Psiscq^{m,k,\ell}(M)$ decomposes microlocally into $\Psisc^{k-\ell,\ell}(M)$  away from the $\fibi$-face and  $\Psisc^{m,k-m}(M)$  away from the $\basi$-face.   In particular, if $k=m+\ell$, we conclude that
  $$
 \Psisc^{m,\ell}(M)=\Psiscq^{m,m+\ell,\ell}(M).
 $$
 
In analogy to the scattering calculus, the \emph{microsupport} $\wf'_\scq(A)\subset \cores$ of $A\in \Psi^{m,k,\ell}_\sc(M)$ is the complement of the set of points $q\in \p\cores$ at which the symbol of $A$ coincides in a neighborhood of $A$ with a symbol in {$S^{-N,-K,-L}(\Tsc^*M)$} for all $N,K,L\in\rr$ (or equivalently, in $S^{-N,-L}(\Tsc^*M)$ for all $N,L\in \rr$). The definition applies to $\Psisc^{m,\ell}(M)=\Psiscq^{m,m+\ell,\ell}(M)$ in particular.  The important properties are that $\WF'_{\scq}(A)=\emptyset$ implies $A\in\cW^{-\infty,-\infty}(M)$ (which is immediate from the definition) and that composition, at this point for $\sc$-pseudo-differential operators, is microlocal, namely:
 $$
 \WF'_{\scq}(A_1 A_2)\subset\WF'_{\scq}(A_1)\cap\WF'_{\scq}(A_2)
 $$ 
 for $A_i\in\Psisc^{m_i,\ell_i}(M)$ ($i=1,2$), as follows immediately from the composition rule in the scattering algebra.

 Then, an equivalent way  of phrasing the definition of $\Psiscq^{m,k,\ell}(M)$  is that an operator $A$ acting on, say, Schwartz functions is in $\Psiscq^{m,k,\ell}(M)$ if and only if  $A\in\Psisc^{\max(m,k-\ell),\ell}(M)$ (or indeed $A\in\Psisc^{M,K,L}(M)$ for some $M,K,L$) and for all $Q_{\fibi},Q_{\basi}\in\Psisc^{0,0}(M)=\Psiscq^{0,0,0}(M)$ with $\WF'_{\scq}(Q_{\fibi})\cap\basi=\emptyset$, $\WF'_{\scq}(Q_{\basi})\cap\fibi=\emptyset$,
 \beq\label{eq:QA}
 Q_{\fibi}A\in\Psisc^{m,k-m}(M) \mbox{ and  } Q_{\basi}A\in\Psisc^{k-\ell,\ell}(M), 
 \eeq
 where the compositions are well-defined in the  $\sc$-algebra. The latter is equivalent to 
 \beq\label{eq:AQ}
  AQ_{\fibi}\in\Psisc^{m,k-m}(M) \mbox{ and  } AQ_{\basi}\in\Psisc^{k-\ell,\ell}(M),
 \eeq
 and it suffices to check \eqref{eq:QA} or \eqref{eq:AQ} for any fixed $Q_{\basi},Q_{\fibi}$ with $Q_{\basi}+Q_{\fibi}=\one$  and  $\WF'_{\scq}(Q_{\fibi})\cap\basi=\emptyset$, $\WF'_{\scq}(Q_{\basi})\cap\fibi=\emptyset$. For instance to see that \eqref{eq:QA} implies {eq:AQ}, if $Q_{\fibi}$ is as above, choose $Q'_{\fibi}$ to have similar properties but in addition $\WF'_{\scq}(I-Q'_{\fibi})\cap\WF'_{\scq}(Q_{\fibi})=\emptyset$, and then the identity $AQ_{\fibi}=(Q'_{\fibi}A)Q_{\fibi}+(I-Q'_{\fibi})AQ_{\fibi}$ shows the claim for the second term is in $\cW_\sc^{-\infty,-\infty}(M)$ by wave front set considerations (which only use $\WF'_{\scq}$ properties of $\sc$-operators).

In the resolved calculus, the \emph{principal symbol} $\sigma_{\scq}(A)$ of  $A\in \Psiscq^{m,k,\ell}(M)$ is the equivalence class of its symbol in
 $$
S^{m,k,\ell}(\Tsc^*M)/S^{m-1,k-2,\ell-1}(\Tsc^*M),
 $$
in particular we quotient out symbols which are  {\em two} orders better at the front face $\ff$. This simply
 corresponds to the scattering algebra in $\RR^n$ gaining both a factor
 of $\langle z\rangle^{-1}$ and $\langle\zeta\rangle^{-1}$ in
 asymptotic expansions, and both
 vanish at $\ff$. For this reason the symbol short exact sequence takes the form
 $$
 0\to\Psiscq^{m-1,k-2,\ell-1}(M)\to\Psiscq^{m,k,\ell}(M)\to
 S^{m,k,\ell}(\Tsc^*M)/S^{m-1,k-2,\ell-1}(\Tsc^*M)\to 0,
 $$
where the second arrow is the embedding  of $\Psiscq^{m-1,k-2,\ell-1}(M)$ in $\Psiscq^{m,k,\ell}(M)$ and the third arrow is the principal symbol map $\sigma_\scq$.

Correspondingly, we  have a notion of elliptic set on the
 boundary of the resolved phase space:
 $$
\Ell_\scq(A)\subset \cores.
 $$
 If $A$ is classical, then its \emph{elliptic set} is the  complement $\elll_\scq(A)=\p\cores\setminus \Char_\scq(A)$ of the \emph{characteristic set} $\Char_\scq(A)$, defined as the closure of the zero set of the principal symbol $\sigma_\scq(A)$; this then generalizes to non-classical symbols by appealing to a classical representative.  Finally, $A$ is said to be \emph{elliptic} in $\Psi^{m,k,\ell}_\scq(M)$ if $\elll_\scq(A)=\p\cores$ (or equivalently, if $\Char_\scq(A)=\emptyset$).  Since we are actually interested in proving propagation estimates along the Hamilton flow in the $\rm sc$-sense, it is convenient to map the microsupport and elliptic set back to $\p\co$  via the blow-down map $\beta$ and introduce the following short-hand notation.
 
 \begin{definition} For $A\in  \Psi^{m,k,\ell}_\scq(M)$, we denote $$
 \wf'_\sc(A):=\beta(\wf'_\scq(A)), \quad  
 \elll_\sc(A)\defeq \beta (\elll_\scq(A)).$$
 \end{definition}
 
 This is  consistent with the notation already in use, i.e., if $A\in \Psisc^{m,\ell}(M)$ then the the sets $\wf'_\sc(A)$ and $\elll_\sc(A)$ defined in the $\sc$-calculus sense coincide with the ones defined above.
 
 There is a  notion of weighted Sobolev spaces, defined similarly to the $\rm sc$-setting.  Concretely, for $s,r,l\in\rr$, we define
   $$
   H^{s,r,l}_\scq(M)= \{ u \in H_\scq^{-\infty,-\infty,-\infty}(M) \st A u \in L^2(M)\},
   $$
   where $A\in\Psiscq^{s,r,l}(M)$ is some fixed invertible classical elliptic operator and 
   $$
   H_\scq^{-\infty,-\infty,-\infty}(M) = \textstyle\bigcup_{m,k,\ell} \Psiscq^{m,k,\ell}(M) L^2(M). 
   $$
    The norm of $u\in   H^{s,r,l}_\scq(M)$  is by definition $\| u\|_{s,r,l}= \|A u \|$. The relationships between the $\sc$-calculus and its resolved version imply that microlocally away from the $\fibi$-face, $H^{s,r,l}_\scq(M)$ is the same as $H^{r-l,l}_\sc(M)$, and microlocally away from  the $\basi$-face, $H^{s,r,l}_\scq(M)$ is $H^{s,r-s}_\sc(M)$.

 The resolved pseudo-differential class has good composition and mapping properties:
     
 \begin{proposition}\label{respdo} For all $m_1,k_1,\ell_1\in\rr$ and all $m_2,k_2,\ell_2\in\rr$,
 $$
 \Psiscq^{m_1,k_1,\ell_1}(M)\circ   \Psiscq^{m_2,k_2,\ell_2}(M) \subset  \Psiscq^{m_1+m_2,k_1+k_2,\ell_1+\ell_2}(M),
 $$
 and the principal symbol map $\sigma_\scq$ is a filtered $*$-algebra homomorphism.
 Furthermore, for all  $m,k,\ell\in \rr$ and  all $s,r,l\in \rr$, 
 $$
\Psiscq^{m,k,\ell}(M)\subset B( H_\scq^{s,r,l}(M), H_\scq^{s-m,r-k,l-\ell}(M)).
 $$ 
 \end{proposition}
 \begin{proof} This follows from the characterization using microlocalizers $Q_{\fibi}$ and $Q_{\basi}$, see  \eqref{eq:QA},   since $\Psisc(M)$ itself is a filtered $*$-algebra. Indeed, for instance, to show that $$
      Q_{\fibi}AB\in\Psisc^{m_1+m_2,k_1+k_2-m_1-m_2}(M)
      $$
      for $A\in\Psiscq^{m_1,k_1,\ell_1}(M)$ and $B\in\Psiscq^{m_2,k_2,\ell_2}(M)$, let $Q'_{\fibi}$ be such that $\WF'_{\scq}(\one-Q'_{\fibi})\cap\WF'_{\scq}(Q_{\fibi})=\emptyset$; then
      $$
      Q_{\fibi}AB=(Q_{\fibi}A)(Q'_{\fibi}B)+(Q_{\fibi}A(\one-Q'_{\fibi}))B
      $$
      is in the stated space for the last term is in $\cW_\sc^{-\infty,-\infty}$ by $\WF'_{\scq}$ considerations.
 \end{proof}

In particular, operators in $\Psisc^{0,0}(M)=\Psiscq^{0,0,0}(M)$ are bounded on $H_\scq^{s,r,l}(M)$ for all $s,r,l\in \rr$. Another direct consequence of the first part of Proposition \ref{respdo} is that taking commutators increases the a priori $\ff$-order by $2$:
 $$
[\Psiscq^{m_1,k_1,\ell_1}(M), \Psiscq^{m_2,k_2,\ell_2}(M)] \subset \Psiscq^{m_1+m_2-1,k_1+k_2-2,\ell_1+\ell_2-1}(M).
 $$

 \subsection{Propagation and radial estimates in  \texorpdfstring{$\sc$}{sc}-calculus}\label{ss:pe} The estimates proved in later sections can be seen as refinements of the propagation of singularities theorem and of  radial estimates in the $\sc$-setting, which we now recall following \cite{vasygrenoble,vasyessential}.

 Let $P\in \Psi^{2,0}_\sc(M)$ be classical, and suppose
 \beq\label{eq:ImPsp}
P-P^*\in  \Psi^{1,-1-\delta}_\sc(M).
\eeq
for some $\delta>0$. In particular, the principal symbol $p$ of $P$ is real.  
For the moment we do not need to assume that $P$ is a wave or Klein--Gordon operator. To keep the notation as simple as possible we only consider the scalar case, and comment on the generalization to the vector bundle case in a later section.
 
  We are interested in uniform estimates for the family of operators $P-\lambda\in \Psi^{2,0}_\sc(M)$, where $\lambda\in \cc$ is in a closed (not necessarily bounded)   subset
  $$
  \complex_{\varepsilon}^{+}\subset \{ z \in \cc  \st \Im z \geq \varepsilon \}
$$
   of the upper half-plane, separated from the real line by at least $\varepsilon>0$. The presentation in \cite{vasygrenoble} allows for  orders different   than $(2,0)$ and for more general imaginary part, here however we do not discuss these possibilities for the sake of brevity. Note that in our particular situation, $\Im \lambda>0$, and of course $\lambda\in \Psi^{0,0}_\sc(M)$ as a multiple of the identity. The principal symbol of $P-\lambda$ at base infinity $\basinf$ is therefore $p-\lambda$, with $p$ real and $\lambda\neq 0$ imaginary, so $P-\lambda$ is elliptic at $\basinf$ (away from the corner). On the other hand at fiber infinity, $\lambda$ does not enter the principal symbol, so the characteristic set is 
   $$
   \Char_\sc(P-\lambda)=\Char_\sc(P)\cap \fibinf \eqdef \Char_0.
   $$

In this setting, we first state the propagation of singularities theorem, due to Melrose \cite{melrosered} for  fixed $\lambda$,  and to Vasy \cite{vasygrenoble,vasyessential} in full generality, building on Hörmander's classical theorem.  Below, we write $q\sim q'$ if $q$ and $q'$ are connected by a bicharacteristic in $\Char_0$, and we denote {the closed bicharacteristic segment from $q$ to $q'$} by $\gamma_{q\sim q'}\subset \Char_0$. The notation $q\succ q'$ means that $q\sim q'$  and $q$ comes after $q'$ along the flow. Capital letters are used to denote weighted Sobolev orders which can be taken arbitrarily negative, so terms like $\err{u}$ on the right hand side of an inequality are interpreted as microlocally irrelevant errors.

   \bep[Propagation of singularities]\label{pos} Let $s,\ell\in\cf(\co)$ be non-decreasing along the Hamilton flow. Let $B_1,B_2,A_0\in\Psi^{0,0}_{\sc}(M)$ be such that $\wf'_\sc(B_1)\subset \elll_\sc(A_0)$ and  the following control condition is satisfied:
   \beq\label{eq:forward}
   \bea
   &\forall q\in \wf'_\sc(B_1)\cap \Char_0, \\ &\exists\, q'\in\elll_\sc(B_2) \mbox{ s.t. } q\succ q' \mbox{ and } \gamma_{q\sim q'} \subset \elll_\sc(A_0).  
  \eea
   \eeq
Then for all $u\in H^{\SL}_\sc(M)$ and  $\lambda \in \complex_{\varepsilon}^{+}$,
\beq\label{eq:thepos}
   \norm{B_1 u}_{s,\ell} + \module{\Im \lambda}^{\12}    \norm{B_1 u}_{s-\12,\ell+\12}  \lesssim    \norm{B_2 u}_{s,\ell} +  \| A_0 (P-\lambda) u \|_{s-1,\ell+1}+ \err{u}.
   \eeq
   \eep
   
More precisely, by this we mean that there exists $C>0$ such that for all $u\in H^{\SL}_\sc(M)$, if $B_2 u \in\Hsc{s,\ell}$ and $\{A_0(P-\lambda)u\}_{\lambda\in \complex_{\varepsilon}^{+}}$ is bounded in $H^{s-1,\ell+1}_\sc(M)$, then $B_1 u\in \Hsc{s,\ell} \cap \Hsc{s-1/2,\ell+1/2}$, and the estimate
$$
      \norm{B_1 u}_{s,\ell} + \module{\Im \lambda}^{\12}    \norm{B_1 u}_{s-\12,\ell+\12}  \leq C(    \norm{B_2 u}_{s,\ell} +  \| A_0 (P-\lambda) u \|_{s-1,\ell+1}+ \err{u})
$$
holds true (in particular it is uniform in $\lambda \in \complex_{\varepsilon}^{+}$). This  notational convention will be used throughout the text without further mention.

In brief, Proposition \ref{pos} says that knowledge about $u$ being microlocally in $H^{s,\ell}_\sc(M)$   can be propagated {forward} along the flow, i.e.~from $\elll_\sc(B_2)$ to $\elll_\sc(B_1)$, provided that  $(P-\lambda)u$ is in  $H^{s-1,\ell+1}_\sc(M)$ microlocally in a neighborhood of the  trajectory from $\elll_\sc(B_2)$ to $\wf'_\sc(B_1)$.
  

Next, for the radial estimates,  we need a classical dynamics with sinks and/or sources. Here we specialize to the following specific situation (particularly relevant for the wave operator with imaginary spectral parameter).


   \begin{definition}\label{sinkssources} We say that $L_-\subset \basinf$, resp.~$L_+\subset \basinf$, are \emph{sources}, resp.~\emph{sinks} at infinity, if $L_\pm$ is a smooth manifold transversal to $\corner$ and within $\Char_0$:
   \ben 
   \item[a)] $d{p}\neq 0$ on $L_\pm$ and $\xoverline{H}_{p}$ is tangent to $L_\pm$,
   \item[b)]\label{sscond1} there exists a quadratic defining function $\rho_{\pm}$ of $L_\pm$ and a smooth function $\beta_{L_\pm}>0$ satisfying
  $$
   \mp\xoverline{H}_{p} \rho_{\pm}= \beta_{L_\pm} \rho_{\pm} + s_\pm + r_\pm  
  $$
   for some smooth $s_\pm,r_\pm$ such that $s_\pm>0$  on $L_\pm$ and $r_\pm$ vanishes cubically at $L_\pm$,
   \item[c)]\label{sscond2} there exists $\beta_{\pm}\in\cf(\co)$ such that $\beta_{\pm}>0$ on $L_\pm$ and
   $\mp\xoverline{H}_{p} \rho = \beta_{\pm} \rho$ modulo cubically vanishing terms.
      \item[d)]\label{sscond3} there exists $\beta_{\infty}\in\cf(\co)$ such that {$\beta_{\infty,\pm}=0$} on $L_\pm$ and
      $\mp\xoverline{H}_{p} \rho_\infty = \beta_{\infty,\pm}\beta_\pm \rho_\infty$ modulo cubically vanishing terms.
   \een
   \end{definition}

By saying that $\rho_{\pm}\in \cf(\co)$ is a \emph{quadratic defining function of $L_\pm$} one means that $\rho_{\pm}=\sum_i \rho_{\pm,i}^2$ for finitely many  $\rho_{\pm,i}$ such that  $L_\pm=\cap_i\{ \rho_{\pm,i} =0\}$ within $\Char_0 \cap \,\basinf$, and the differentials $d \rho_{\pm,i}$ are  linearly independent  on $L_\pm\cap \Char_0$. 

 A more detailed discussion of  conditions b)--d) (or strictly speaking, of their analogue in the case of sources and sinks located at fiber infinity) can be found in \cite[\S 5.4.7]{vasygrenoble}. For the sake of simplicity,  Definition \ref{sinkssources} is formulated in terms of the  boundary functions $\rho,\rho_\infty$,  but any other  boundary defining functions could be used instead.

The following \emph{lower decay radial estimate} can be used to propagate regularity and decay properties of $u$ into a sink $L_+$ from a punctured neighborhood $U\setminus U_1$. The \emph{higher decay radial estimate} serves to gain regularity and  decay properties in a neighborhood of a source $L_-$ provided that the decay is already better than the \emph{threshold value} $-\12$.
 
 
 \bep[Lower decay radial estimate  {\cite[(5)]{vasyessential}}]\label{radial2}  Let $s\in\rr$ and assume that $\ell$ is non-decreasing along the Hamilton flow and $\ell|_{L_+}<-\12$. Let $B_1,B_2,A_0\in\Psi^{0,0}_{\sc}(M)$ and let $U_1, U$ be open neighborhoods of $L_+$ in  $\p\co$ and assume $U_1\subset \elll_\sc(B_1)$, $\wf'_\sc(B_1)\subset \elll_\sc(A_0)\subset U$ and $\wf'_\sc(B_2)\subset U\setminus U_1$. Assume  that:
    $$ 
    \bea
    &\forall q\in \wf'_\sc(B_1)\cap \Char_\sc(P-\lambda)\setminus L_+, \\ &\exists\, q'\in\elll_\sc(B_2)  \mbox{ s.t. } q\succ q' \mbox{ and }  \gamma_{q\sim q'} \subset \elll_\sc(A_0).  
    \eea
    $$
Then for all $u\in H^{\SL}_\sc(M)$ and $\lambda\in\complex_{\varepsilon}^{+}$,
    $$
    \norm{B_1 u}_{s,\ell} +  \module{\Im \lambda}^{\12}    \norm{B_1  u}_{s-\12,\ell+\12}    \lesssim   \norm{B_2 u}_{s,\ell} +  \| A_0(P-\lambda) u \|_{s-1,\ell+1}+ \err{u}.
    $$
 \eep

    \bep[Higher decay radial estimate {\cite[(4)]{vasyessential}}]\label{radial1}  Let $s\in\rr$ and assume that $\ell$ is non-decreasing along the Hamilton flow and  $\ell|_{L_-}>-\12$. Let $B_1,A_0\in\Psi^{0,0}_{\sc}(M)$ and let $U$ be a sufficiently small open neighborhood of $L_-$ in $\p\co$. Assume   $L_-\subset \elll_\sc(B_1)$ and $\wf'_\sc(B_1)\subset \elll_\sc(A_0)\subset U$. 
     Then for all $s'\in \rr$, $\ell'\in\open{-\12,\ell}$ all $u\in H^{\SL}_\sc(M)$ such that $A_0 u \in\Hsc{s',\ell'}$ and all $\lambda\in\complex_{\varepsilon}^{+}$,
       \beq\label{erff}
       \norm{B_1 u}_{s,\ell} +  \module{\Im \lambda}^{\12}    \norm{B_1 u}_{s-\12,\ell+\12}    \lesssim     \| A_0(P-\lambda) u \|_{s-1,\ell+1}+ \err{u}.
       \eeq
    \eep
    
    One of our objectives is to prove a refinement of these estimates  in the resolved setting. To that end we first recall what are the relevant steps in the proof of  Propositions  \ref{pos}--\ref{radial1} in \cite{vasygrenoble,vasyessential},  and then in the next subsection we explain how exactly the individual steps need to be modified in the resolved situation. 
    
    \smallskip
    \noindent \textbf{Sketch of proof of Proposition  \ref{pos} and Propositions \ref{radial2}--\ref{radial1}.} 
    Let us recall that the core of the proof is a microlocal positive commutator argument. Namely, given $s$ and $\ell$, one first constructs an operator $A\in \Psisc^{2s-1,2\ell+1}(M)$ (often called the \emph{commutant}), such that the  commutator expression
         \beq\label{eq:comm}
         \bea
         i(AP - P^* A) &= i [A,P] + i (P-P^*)A \\
         & = i [A,P] + 2 (\Im  P)A
         \eea
         \eeq 
         has positive principal symbol microlocally in the region of interest. This positivity property  translates then to an \emph{a priori} estimate, which then needs to be supplemented with a few more arguments   to obtain estimates in the form stated in the proposition. More precisely, the proof consists of the following steps.     
         
     \step{1} Anticipating the need for a generalization in \emph{Step 3.}~of the proof, we assume here that $2 \Im P \in \Psisc^{1,-1}(M)$ (rather than  $\Psisc^{1,-1-\delta}(M)$) and we denote  by $\tilde p\in S^{1,-1}(\Tsc^*M)$  its principal symbol. The first task is to construct $A$ on the principal symbol level. Observe that if $a\in S^{2s-1,2\ell+1}(\Tsc^*M)$ is the principal symbol of $A$, then 
    the principal symbol of \eqref{eq:comm} is $-(H_p a  + \tilde p a)\in S^{2s,2\ell}(\Tsc^*M)$. Therefore, the objective is to find $a\in S^{2s-1,2\ell+1}(\Tsc^*M)$  such that
\beq\label{eq:Hpa}
-({H}_p a + \tilde p a) =  b^2 - e 
\eeq
for some real-valued symbol $b\in S^{s,\ell}(\Tsc^*M)$ non-vanishing in the region of interest and some  $e\in  S^{2s,2\ell}(\Tsc^*M)$  supported in the region where we have a priori knowledge about $u$ being in microlocally $\Hsc{s,\ell}$.

 Let us first consider the case of propagation of singularities, i.e.~Proposition  \ref{pos}.  In practice it suffices to solve \eqref{eq:Hpa} in terms of symbols supported in a small neighborhood $U\subset \co$ of each $q\in \p\co$ in the region of interest, and then construct $a$ as a locally finite sum.  In what follows we write $c\in S^{0,0}(U)$ if $c\in S^{0,0}(\Tsc^*M)$ and $\supp c \subset U$ (and similarly for more general orders).  

To make the connection with the dynamics of the rescaled Hamilton flow $\xoverline{H}_p$ on $\co$, it is useful to reduce \eqref{eq:Hpa} to a similar problem for symbols of order $(0,0)$.  Namely, on the level of $(0,0)$-order symbols and for sufficiently small $U$, one shows that for all  $C \geq 0$ there exists $a_0,b_{0},e_{0}\in S^{0,0}(U)$ with the desired support properties (and with $a_0\geq 0$),   such that  
 \beq\label{eq:Hpa2}
-\xoverline{H}_p a_0 =  C a_0  +b^2_0 - e_0.
 \eeq 
Once the problem is formulated in this way, the construction is rather standard and its details are not significant for us, see \cite[\S5.4.3]{vasygrenoble}. On the other hand, it is important to recall how  \eqref{eq:Hpa} is deduced from the $(0,0)$-order version \eqref{eq:Hpa2}. One considers separately the two  relevant cases:
   \ben 
   \item[a)] If $q\in \corner$ then we set $a\defeq \rho^{-2s+1}_\infty \bdf^{-2\ell-1} a_0\in S^{2s-1,2\ell+1}(U)$.  Recall that $\xoverline{H}_p$ is tangent to $\p\co$, so $\xoverline{H}_p \rho=0$ and $\xoverline{H}_p\rho_\infty=0$ at $q$. Supposing for the moment that $s,\ell$ are constant,  this implies that $\xoverline{H}_p  \rho^{-2s+1}_\infty \bdf^{-2\ell-1}=0$ at $q$ and in consequence, modulo some vanishing term, $H_p a$ is a multiple of  $\xoverline{H}_p a_0$ times powers of $\rho$ and $\rho_\infty$. In this situation it is straightforward to conclude \eqref{eq:Hpa} from \eqref{eq:Hpa2} by  multiplying everything with appropriate powers of $\rho$ and $\rho_\infty$ and by taking $C$ large enough so that 
   $$
   \tilde p a + C \rho_\infty^{-1} \rho a \geq 0,
   $$
which can then be absorbed into the definition of $b^2$.
    More generally if the orders $s,\ell$ are variable, then $\xoverline{H}_p  \rho^{-2s+1}_\infty \bdf^{-2\ell-1}$ does not necessarily vanish, in fact:
  \beq\label{eq:der}
  \bea
  &\xoverline{H}_p  \rho^{-2s+1}_\infty =2 \rho^{-2s+1}_\infty (-\log \rho_\infty)   (-\xoverline{H}_p s),\\
   &\xoverline{H}_p  \rho^{-2\ell-1} = 2\rho^{-2\ell-1} (-\log \rho)   (-\xoverline{H}_p \ell)
  \eea
  \eeq
  modulo $O(\rho)O(\rho_\infty)$.
  However, we have assumed $\xoverline{H}_p s\leq 0$ and $\xoverline{H}_p \ell \leq 0$, so the extra contributions from the derivatives \eqref{eq:der} are non-negative and thus can be absorbed in the definition of $b^2$.

   \smallskip  
   \item[b)] If $q\in\fibinf \setminus \corner$ then we set  $a\defeq \rho_\infty^{-2s-1} a_0\in S^{2s+1,0}(U)$. For  $U$  sufficiently small so it does not intersect the corner,  the order at base infinity is irrelevant and we can write $\rho_\infty^{-2s-1} a_0\in S^{2s-1,2\ell+1}(U)$. Then everything is fully analogous to a), except that  only the  $\rho_\infty^{-2s-1}$ weight enters the arguments. In particular, instead of the two identities \eqref{eq:der}  it suffices to use 
   \beq
\xoverline{H}_p  \rho^{-2s+1}_\infty =2 \rho^{-2s+1}_\infty (-\log \rho_\infty)   (-\xoverline{H}_p s)
   \eeq
modulo $O(\rho_\infty)$.
\een

In the case of radial points,  the above construction no longer applies. Instead, one needs to rely on the specific dynamics at sources and sinks  given by Definition \ref{sinkssources}. 

The commutant $a$ is taken to be of the form
\beq\label{eq:rcomm}
a = \phi(\rho_\pm)  \rho^{-2s+1}_\infty\rho^{-2\ell-1} a_0,
\eeq
 where $\phi\in \cf_\c(\clopen{0,\infty})$ equals $1$ near $0$, $\phi\geq 0$ and $-\phi'\geq 0$ (with smooth square roots).  The symbol  $a_0$ localizes near the characteristic set $\Char_0$ and can be chosen such that $\xoverline{H}_p a_0$ is supported away from $\Char_0$. In consequence, the  elliptic estimate applies on the support of $\xoverline{H}_p a_0$, so $a_0$ contributes to the computation of  $-{H}_p a$ in an unessential way. Concerning  the contribution from $\phi(\rho_\pm)$, by taking $\phi$ supported close enough to $0$ one deduces easily from  the source/sink condition b) in Definition \ref{sinkssources} that $
\pm \xoverline{H}_p\phi (\rho_\pm)  \geq 0
 $
in a neighborhood of $L_\pm$.  More precisely,
$$
\pm \xoverline{H}_p\phi (\rho_\pm)  =c_1^2
$$
for some $c_1$  supported on $\supp \phi'(\rho_\pm)$ (thus $c_1$ is supported in a punctured  neighborhood of $L_\pm$).  Finally, the contributions from $\rho^{-2s+1}_\infty$  and $\rho^{-2\ell-1}$ are computed from conditions c)--d) in Definition \ref{sinkssources} (i.e., $\mp\xoverline{H}_{p} \rho = \beta_{\pm} \rho$ with $\beta_\pm>0$ on $L_\pm$ and  $\mp\xoverline{H}_{p} \rho_\infty =\beta_{\infty,\pm}\beta_\pm \rho_\infty$ with $\beta_{\infty,\pm}= 0$ on $L_\pm$).  Without loss of generality we can assume that $s$ and $\ell$ are constant Considering all this and   disregarding  contributions arising from $a_0$, one obtains
 \beq\label{eq:toexamine}
 -({H}_p a + \tilde p a)  = \mp 2 \rho_\infty^{-2s} \rho^{-2\ell}\Big(  c_1^2 +\phi(\rho_\pm)     \beta_\pm \big(\beta_{\infty,\pm} (s-\textstyle\12)  +  \ell + \tilde\beta_\pm+\12 \big)  \Big), 
 \eeq
 where $\tilde\beta_\pm\in S^{0,0}(\Tsc^*M)$  is a suitable normalization of $\tilde p$, namely $\tilde p  \eqdef \pm 2\beta_\pm \tilde\beta_\pm \rho$ sufficiently close to $L_\pm$. 
 


Let us examine positivity properties of the r.h.s.~of \eqref{eq:toexamine}. Since $\beta_{\infty,\pm}=0$ on $\{ \rho_\pm =0\}$, the term with $\beta_{\infty,\pm}$ can be made as small as wanted by choosing $\supp \phi$ sufficiently close to $0$. Thus, it can be disregarded if $\ell  + {\tilde\beta_+}\neq 0$. We consider  the following two cases:
\ben
\item  In the sink case ($+$ sign), suppose  $\ell  + {\tilde\beta_+}  <  -\12$ on $L_+$. Then the  second summand in \eqref{eq:toexamine}  is  $\geq 0$ (it contributes to $b^2$), and $>0$ on $L_+$. The first summand is not positive (it contributes to $-e$) and it is supported in a punctured neighborhood of $L_+$. This situation is relevant for Proposition \ref{radial2}.
\item In the source case ($-$ sign), suppose $\ell  + \tilde\beta_- >  -\12$ on $L_-$. Then all terms in \eqref{eq:toexamine}  are  $\geq 0$, and $>0$ on $L_-$ (they contribute to $b^2$, and we can set $e=0$). This situation is relevant for Proposition \ref{radial1}.
\een
\smallskip

\step{2} In the second step we deduce an a priori estimate. Namely, given symbols satisfying
\beq\label{eq:done}
-({H}_p a + \tilde p a) =  b^2 - e,
\eeq
one takes $A\in \Psisc^{2s-1,2\ell+1}(M)$, $B\in \Psisc^{s,\ell}(M)$, resp.~$E\in \Psisc^{2s,2\ell}(M)$  with principal symbol $a$, $b$, resp.~$e$. We can arrange that $A$ is the square of a formally self-adjoint operator $A^\12\in\Psisc^{s-\12,\ell+\12}(M) $ with principal symbol $a^\12$.  As already remarked, the l.h.s.~of \eqref{eq:done} is the principal symbol of $i(AP - P^* A)$, and of course the r.h.s.~is the principal symbol of $B^* B - E$, so on the operator level
 $$
 i(AP - P^* A) = B^* B -E  + F
 $$ 
for some lower order  term $F\in \Psisc^{2s-1,2\ell-1}(M)$. Thus, for all $u\in H_\sc^{\infty,\infty}(M)$, 
$$
\bea
\bra i (P-\lambda) u  , Au \ket- \bra i A u, (P-\lambda) u\ket &=  \bra i(AP - P^* A)  u, u \ket + 2 (\Im \lambda) \bra  A  u,u\ket\\
&= \norm{B u}^2 - \bra E u , u \ket +  \bra F u , u \ket  + 2 (\Im \lambda) \bra  A  u,u\ket.
\eea
$$
On the other hand, a routine use of Cauchy--Schwarz inequality and Young inequality gives
\beq\label{eq:cs}
\bea
\module{\bra  (P-\lambda) u  , Au \ket} &\leq  \|A^\12 (P-\lambda)u\|_{-\12,\12}  \|A^\12 u\|  _{\12,-\12}\\
&\leq \epsilon^{-1}   \|A^\12 (P-\lambda)u\|_{-\12,\12}^2  +  \epsilon \|A^\12 u\|_{\12,-\12}^2
\eea
\eeq
for any $\epsilon >0$. Thus, 
$$
\norm{B u}^2  +  2 (\Im \lambda) \bra  A  u,u\ket  \leq  \bra E u , u \ket + \epsilon^{-1}   \|A^\12 (P-\lambda)u\|_{-\12,\12}^2  +  \epsilon \|A^\12 u\|_{\12,-\12}^2 -  \bra F u , u \ket. 
$$
For $\epsilon$ sufficiently small, the $\epsilon \|A^\12 u\|_{\12,-\12}^2$ term can be absorbed into the first term on the l.h.s. Using  elliptic estimates  and rewriting the resulting inequality in terms of norms rather than squared norms, one gets  an a priori estimate of the form
$$
 \norm{B_1 u}_{s,\ell} + \module{\Im \lambda}^{\12}    \norm{B_1 u}_{s-\12,\ell+\12}  \lesssim    \norm{B_2 u}_{s,\ell} +  \| A_0 (P-\lambda) u \|_{s-1,\ell+1}+ \norm{B_3 u}_{s-\12,\ell-\12}+  \err{u},
$$
where  the $\norm{B_3 u}_{s-\12,\ell-\12}$ term  comes from  the lower order term $F$ (and $B_2=0$ in the case of $\ref{radial1}$).

\step{3} The a priori estimate above is true provided that the norms on both the r.h.s.~and l.h.s.~are finite. To get an estimate in the strong sense (i.e.,  if the  norms on the r.h.s.~are finite then the l.h.s.~is finite), one first chooses a suitable family of  invertible elliptic operators $R_\delta$ of sufficiently negative order for $\delta>0$, and such that $R_\delta\to \one$ in $\Psisc^{0+,0+}(M)$.   Then  one applies the a priori estimate with $P$ replaced by $R_\delta P R_\delta^{-1}$, and with $u$ replaced by  its regularized version  $u_\delta=R_\delta u$.  Considering the operator  $R_\delta P R_\delta^{-1}$ instead of $P$ adds an extra imaginary part, but this still fits into the assumptions of  \emph{Step 1} and \emph{Step 2}, and uniformity of the estimates allows to take the $\delta\to 0^+$ limit, see  \cite[\S5.4.4]{vasygrenoble} for the propagation of singularities case.    In the case of radial estimates, a further regularization argument is needed to use integration by part (to trade $P^*$ for $P$ in scalar products), see the paragraph following \cite[(5.62)]{vasygrenoble}.

\step{4} Finally, an iterative use of the estimate  (which uses the assumptions on the dynamics)  shows that   the error term $\norm{B_3 u}_{s-\12,\ell-\12}$ can be eliminated, and with the help of the elliptic estimate (and propagation of singularities, in the radial estimates case) one proves that the final assertion is true for \emph{all} $B_1,B_2,A_0\in \Psi^{0,0}_\sc(M)$ satisfying the assumptions. \qed

     \subsection{Propagation and radial estimates in resolved calculus}\label{ss:pe2} Now moving to the resolved setting, we first prove a propagation of singularities in the sense of $H^{s,k,\ell}_\scq(M)$ regularity. Note that one could attempt to prove various statements referring to an appropriate generalization of the Hamilton flow on the resolved space $\cores$. Here however we will only need estimates along  the rescaled Hamilton flow  in the already discussed $\sc$-sense.

    Recall that our operator of interest is $P-\lambda\in \Psi^{2,0}_\sc(M)$ with $P-P^*\in \Psi^{1,-1-\delta}_\sc(M)$ and with $\Im \lambda >0$, in particular the  characteristic set $\Char_0$ is $\lambda$-independent.
     
       \bep[Propagation of singularities]\label{pos2} Let $s,k,\ell\in\cf(\co)$ and assume $s,k$ are  non-decreasing along the Hamilton flow. Let $B_1,B_2,A_0\in\Psi^{0,0}_{\sc}(M)$ be such that $\wf'_\sc(B_1)\subset \elll_\sc(A_0)$ and:
       \beq\label{eq:forward}
       \bea
       &\forall q\in \wf'_\sc(B_1)\cap \Char_0, \\ &\exists\, q'\in\elll_\sc(B_2) \mbox{ s.t. } q\succ q' \mbox{ and } \gamma_{q\sim q'} \subset \elll_\sc(A_0).  
      \eea
       \eeq
  Then for all $u\in H^{\SKL}_\scq(M)$ and $\lambda \in \complex_{\varepsilon}^{+}$,
    \beq\label{eq:thepos2}
       \norm{B_1 u}_{s,k,\ell} + \module{\Im \lambda}^{\12}    \norm{B_1 u}_{s-\12, k,\ell}  \lesssim   \norm{B_2 u}_{s,k,\ell} +  \| A_0 (P-\lambda) u \|_{s-1,k,\ell}+ \errr{u}.
       \eeq
       \eep     
         \begin{proof} We  modify the proof of  propagation of singularities in the $\sc$-setting, i.e.~Proposition \ref{pos}. One preliminary aspect to take care of is the validity of an analogue of the elliptic estimate in the resolved setting with the asserted behaviour in $\Im \lambda$. We defer that discussion to \sec{s:semicl}, where we introduce semi-classical techniques, and we focus for now on  estimates near the characteristic set.
          
            We do the following  modification in \emph{Step 1.}~of the proof of Proposition \ref{pos}.   When constructing  the principal symbol $a$  of the commutant $A$ near a point $q\in \corner$ at the corner, away from the $\basi$-face  we  use the weight  $\rho^{-2k}_\ff$ instead of $\rho^{-(2s-1)}_\infty \bdf^{-(2\ell+1)}$ . Then, all the considerations at the principal symbol level are fully analogous provided that we find a  replacement of the identities \eqref{eq:der}  to justify positivity of $\xoverline{H}_p \rho_{\ff}^{-2k}$  when $k$ is not constant. 
         
         Indeed, writing $\xoverline{H}_p=
         \beta \bdf\pa_\bdf+ \beta_\infty \rho_\infty\pa_{\rho_\infty}$
        modulo derivatives that
         annihilate $\bdf$ and $\bdf_\infty$ (where $\beta, \beta_\infty \in \cf(\co)$ is such that $\beta=\mp\beta_\pm$ on $L_\pm$ and similarly for $\beta_\infty$, see Definition \ref{sinkssources}), we have 
         $$
        \xoverline{H}_p  (\bdf+\rho_\infty)^{-2k}=2(\bdf+\rho_\infty)^{-2k}\big(k(\beta \bdf +\beta_\infty\rho_\infty) (\bdf+\rho_\infty)^{-1}+(        -\xoverline{H}_p k)(-\log (\bdf +\rho_\infty))\big).
         $$
Since $k$ is non-increasing along the Hamilton flow, $\xoverline{H}_p k\leq 0$, so the second term overrules the  first one for sufficiently small $\rho_\ff=\rho+\rho_\infty$, hence $\xoverline{H}_p \rho_{\ff}^{-2k}\geq 0$ close to $\corner$.

 Then the remaining steps of the proof apply almost verbatim, with the  difference that the commutant $A$ is now in $\Psiscq^{2s-1,2k,2\ell+1}(M)$ instead of $\Psisc^{2s-1,2\ell+1}(M)$,  so we have to use the resolved calculus  instead of the $\sc$-calculus. In particular $P\in \Psiscq^{2,2,0}(M)$,  
 $[A,P]\in \Psiscq^{2s,2k,2\ell}(M)$ (recall that we gain {two} orders of regularity at the front face when taking a commutator),  and $B\in \Psiscq^{s,k,\ell}(M)$. The decay orders are actually irrelevant  away from the $\basi$-face, and at the $\basi$-face  we can use the elliptic estimate   instead of the positive commutator argument (thanks to the assumption $\Im\lambda \geq\varepsilon$), hence the improved decay orders in \eqref{eq:thepos2} (strictly speaking, as a result one obtains an estimate with  $\norm{B_1 u}_{s,k,\ell-\12} + \module{\Im \lambda}^{1/2}    \norm{B_1 u}_{s-\12, k,\ell}$ on the l.h.s.~instead of $\norm{B_1 u}_{s,k,\ell} + \module{\Im \lambda}^{1/2}    \norm{B_1 u}_{s-\12, k,\ell}$, but these are equivalent for $\Im\lambda \geq\varepsilon$).  Note that the formulation of the proposition is ultimately somewhat simplified thanks to the fact that $\Psiscq^{0,0,0}(M)=\Psisc^{0,0}(M)$. 
    \end{proof}

Next, we consider estimates near radial sets. 
The crucial observation is that within the resolved calculus we can consider commutants which are localized in a neighborhood of the $\basi$-face (and away from the $\fibi$-face), or in a neighborhood of the $\fibi$-face (and away from the $\basi$-face). A simple example of a symbol in $S^{0,0,0}(\Tsc^*M)$  localized in the former region is given by    $\chi(\bdf/\rho_\infty)$ for $\chi\in \cf_\c(\clopen{0,\infty})$  supported sufficiently close to $0$. Similarly, $\chi(\rho_\infty/\bdf)\in S^{0,0,0}(\Tsc^*M)$ is supported away from the $\basi$-face. When deriving positive commutator estimates, we take a commutant $a$ of a similar form as in the $\sc$-setting, but involving in addition one of these two cutoffs. In consequence, positivity properties of ${H}_p a$ are affected by extra terms, the signs of which depends on  the sign of $\chi$, $\chi'$ and $\xoverline{H}_p(\bdf/\rho_\infty)$. In regard to the latter, we simply have
\beq\label{eq:2rho}
\mp \xoverline{H}_p(\bdf/\rho_\infty)=( 1- \beta_{\infty,\pm}) \beta_\pm\bdf/\rho_\infty, \quad \pm \xoverline{H}_p(\rho_\infty/\bdf)=( 1- \beta_{\infty,\pm}) \beta_\pm \rho_\infty/\bdf. 
\eeq
Recall that $\beta_\pm=0$ on $L_\pm$, so close to the radial sets this has a behaviour  analogous  to the sink/source property $\mp \xoverline{H}_p \bdf=\beta_\pm \bdf$.

We first prove a radial estimate localized near the $\fibi$-face.   In what follows, we use the same notation for $L_\pm$ and its  lift to $\p\cores$.

   \bel[Radial estimates near $\fibi$-face]\label{lem:rad1} Let $s,k\in\cf(\co)$ be non-decreasing along the Hamilton flow.  Let $U_1,U_2,U$ be sufficiently small neighborhoods of the $\fibi$-face in $\p\cores$ and suppose $B_1,B_2,A_0\in\Psi^{0,0}_{\sc}(M)$ satisfy  $U_1\subset \elll_\scq(B_1)$, $\wf'_\scq(B_1)\subset \elll_\scq(A_0)\subset U$,  $\wf'_\scq(B_2)\subset U_2$, and
   $$ 
            \bea
            &\forall q\in \wf'_\sc(B_1)\cap \Char_0\setminus L_+, \\ &\exists\, q'\in\elll_\sc(B_2)  \mbox{ s.t. } q\succ q' \mbox{ and }  \gamma_{q\sim q'} \subset \elll_\sc(A_0).  
            \eea
   $$ 
   Suppose one of the following is true:
\ben
 \item  $(k-s)|_{L_+}<-\12$,  $U_1 \subset \ff^\inti$, and $U_2$ is a small punctured neighborhood of $L_+$,
  \item  $(k-s)|_{L_-}>-\12$,  $U_2\subset  \ff^\inti$, and $U_2$ is in a small neighborhood of $L_+$.
\een
Then for all $u\in H^{\SKL}_\scq(M)$ and  $\lambda\in\complex_{\varepsilon}^{+}$,
         $$
         \norm{B_1 u}_{s,k,*} +  \module{\Im \lambda}^{\12}    \norm{B_1  u}_{s-\12,k,*}    \lesssim   \norm{B_2 u}_{s,k,*} +  \| A_0(P-\lambda) u \|_{s-1,k,*}+ \norm{u}_{\scriptscriptstyle{S},\scriptscriptstyle{K},*}.
         $$        
   \eel
   
   The $*$ notation above means that the estimate is  true for all choices of decay order (as trivially follows from the assumption that $B_1,B_2,A_0$ have wavefront set away from the $\basi$-face).
   
   \refproof{Lemma \ref{lem:rad1}} We use a positive commutator argument  similar to the proof of Proposition \ref{radial2},  but multiply by a $\chi (\rho_\infty/\bdf )$ cutoff the commutant \eqref{eq:rcomm} in \emph{Step 1.}, i.e.~we take
   $$
   a = \chi (\rho_\infty/\bdf ) \phi(\rho_\pm)  \rho^{-2s+1}_\infty\rho^{-2\ell-1} a_0
   $$
with $\ell=k-s$, for some cutoff function $\chi\in \cf_\c(\clopen{0,\infty})$ equal $1$ near $0$, $\chi\geq 0$ and $-\chi'\geq 0$, with smooth square roots $\chi^\12$ and $(-\chi')^\12$. Then by \eqref{eq:2rho}, 
$$
{\mp \phi(\rho_\pm)\xoverline{H}_p \chi(\rho_\infty/\bdf)=-\phi(\rho_\pm)\big( 1+O(\rho_\pm)\big)\beta_\pm  \chi'(\rho_\infty/\bdf)\rho_\infty/\bdf \geq 0}
$$
provided $\supp\phi$ is sufficiently  small, so $\mp \phi(\rho_\pm)\xoverline{H}_p \chi(\rho_\infty/\bdf)=c_2^2$ for some real-valued $c_2$ supported in $\supp \chi'(\rho_\infty/\bdf)\cap \supp \phi(\rho_\pm)$ (thus in a punctured neighborhood of $\fibi$).  Disregarding the symbol $a_0$ which localizes near the characteristic set $\Char_0$, we conclude that $ -({H}_p a + \tilde p a) $ equals
 \beq\label{eq:toexamine2}
 \bea
  \mp 2\rho_\infty^{-2s} \rho^{-2k+2s}\Big(   \chi (\rho_\infty/\bdf )  c_1^2 - c_2^2  + \phi(\rho_\pm)  \chi (\rho_\infty/\bdf )    \beta_\pm \big(   k-s + \tilde\beta_\pm  + \textstyle\12 + O(\rho_\pm) \big) \Big),
 \eea
 \eeq
 where $c_1$ is  supported on $\supp \phi'(\rho_\pm)$ (thus    in a punctured  neighborhood of $L_\pm$).
 We consider the  following two cases:
 \ben
 \item  In the sink case, $L_+$, suppose  $k-s + \tilde\beta_- <  -\12$ on $L_+$. Then the second+third summand in \eqref{eq:toexamine2}  is $\geq 0$, and $>0$ on a subset of the interior of the $\ff$-face  (intersected with a small neighborhood of $L_{+}$). The remaining   one  is supported  in a neighborhood of the $\fibi$-face (intersected with a punctured neighborhood of $L_{+}$).  \smallskip
 \item In the source case, $L_-$, suppose $k-s  + \tilde\beta_- >  -\12$ on $L_-$. Then the first+third term in \eqref{eq:toexamine}  is $\geq 0$, and $>0$ on a neighborhood of the $\fibi$-face   (intersected with a neighborhood of $L_-$).  The remaining   one  is supported in the interior of the $\ff$-face  (intersected with a small  neighborhood of $L_+$). 
 \een
From that point on the proof is analogous to Proposition \ref{radial2} (again, the positive summands on the r.h.s.~of used to define $b^2$, and the other ones absorbed in the definition of $e$).
\qeds

Next, we show estimates localized near the $\basi$-face, and then combine them with the previous lemma to get radial estimates in $\Hscq{s,k,\ell}$ spaces, summarized in  Propositions \ref{radial3}--\ref{radial4}   below. A crucial  feature of the resolved setting is that on $\ff^\inti$, $\Im \lambda$ is \emph{sub-principal} and  therefore the assumption $\Im \lambda\geq \varepsilon>0$ can be usefully exploited in the positive commutator estimates.
       
   \bep[Lower decay radial estimate]\label{radial3}  Let $s,k,\ell\in\cf(\co)$ and assume that $s,k,\ell$ are non-decreasing along the Hamilton flow and $(k-s)|_{L_+}<-\12$. Let $B_1,B_2,A_0\in\Psi^{0,0}_{\sc}(M)$ and let $U_1, U$ be open neighborhoods of $L_+$ in  $\p\co$. Assume that $U_1\subset \elll_\sc(B_1)$, $\wf'_\sc(B_1)\subset \elll_\sc(A_0)\subset U$, $\wf'_\sc(B_2)\subset U\setminus U_1$, and:
      $$ 
      \bea
      &\forall q\in \wf'_\sc(B_1)\cap \Char_0\setminus L_+, \\ &\exists\, q'\in\elll_\sc(B_2)  \mbox{ s.t. } q\succ q' \mbox{ and }  \gamma_{q\sim q'} \subset \elll_\sc(A_0).  
      \eea
      $$
Then for all $u\in H^{\SKL}_\scq(M)$ and $\lambda\in\complex_{\varepsilon}^{+}$, 
      \beq\label{eq:newrad1}
      \norm{B_1 u}_{s,k,\ell} +  \module{\Im \lambda}^{\12}    \norm{B_1  u}_{s-\12,k,\ell}    \lesssim   \norm{B_2 u}_{s,k,\ell} +  \| A_0(P-\lambda) u \|_{s-1,k,\ell}+ \errr{u}
      \eeq
   \eep
   
       \bep[Higher decay radial estimate]\label{radial4}  Let $s,k,\ell\in\cf(\co)$ and assume that $s,k,\ell$ are non-decreasing along the Hamilton flow and  $(k-s)|_{L_-}>-\12$. Let $B_1,A_0\in\Psi^{0,0}_{\sc}(M)$ and let $U$ be a sufficiently small open neighborhood of $L_-$ in $\p\co$. Assume   $L_-\subset \elll_\sc(B_1)$ and $\wf'_\sc(B_1)\subset \elll_\sc(A_0)\subset U$. 
        Then for all $s',k',\ell'\in \rr$ with $k'-s'\in\open{-\12,k-s}$, for all $u\in H^{\SKL}_\scq(M)$ such that $A_0 u \in\Hsc{s',k',\ell'}$, and all  $\lambda\in\complex_{\varepsilon}^{+}$,
          \beq\label{eq:newrad2}
          \norm{B_1 u}_{s,k,\ell} +  \module{\Im \lambda}^{\12}    \norm{B_1 u}_{s-\12,k,\ell}    \lesssim     \| A_0(P-\lambda) u \|_{s-1,k,\ell}+ \errr{u}.
          \eeq
       \eep

    \refproof{Propositions \ref{radial3}--\ref{radial4}} As already outlined, we first consider radial estimates near the $\basi$-face. To that end we write $P-\lambda= (P - i {\varepsilon}/2) -(\lambda-i \varepsilon/2)$ and apply the positive commutator argument to $P-i \varepsilon/2$ instead of $P$, and  to $\lambda-i \varepsilon/2$ instead of $\lambda$. This decomposition has the advantage that now $P-i\varepsilon/2$ has  an imaginary part $\varepsilon/2>0$, which is sub-principal at $\ff^\inti$. On the other hand we still have  $\Im(\lambda-i \varepsilon/2)\geq \varepsilon/2 >0$ for $\lambda \in\complex_{\varepsilon}^{+}$, and the large $\Im \lambda$ behaviour of the estimates will remain unaffected (since $\Im(\lambda-i \varepsilon/2)\geq (\Im \lambda )/2$ on $Z^+_\varepsilon$).   
     
In the  sink/source case $L_\pm$, we take the commutant
       $$
       a = \chi (\bdf/\rho_\infty ) \phi(\rho_\pm)  \rho^{-2(k-\ell)+1}_\infty\rho^{-2\ell-1} a_0,
       $$
       with $\chi$ and $\phi$ as before.
        Similarly as in Lemma \ref{lem:rad1}, we use the sink/source set behaviour \eqref{eq:2rho} to deduce that
         $$
            \pm \phi(\rho_\pm)\xoverline{H}_p \chi(\bdf/\rho_\infty )=-\phi(\rho_\pm)\big( 1+O(\rho_\pm)\big)\beta_\pm  \chi'(\bdf/\rho_\infty )\bdf/\rho_\infty  \eqdef c_1^2
            $$
            is non-negative. Disregarding $a_0$ and derivatives of $s$ and $\ell$, $ -({H}_p a + \tilde p a) $ equals  
        \beq\label{eq:pce}
        \bea
    &  \mp 2\rho_\infty^{-2(k-\ell)} \rho^{-2\ell}\Big(   \chi (\bdf/\rho_\infty )  c_1^2 + c_2^2  + \phi(\rho_\pm)  \chi (\bdf/\rho_\infty)   \big( \beta_\pm (   \ell + \tilde\beta_\pm  + \textstyle\12)   \pm \textstyle\12   \varepsilon  \rho_\infty /\rho + O(\rho_\pm) \big) \Big)\\
    &=  \rho_\infty^{-2(k-\ell)+1} \rho^{-2\ell-1}\Big(  \mp 2c_2^2 \rho/\rho_\infty +   \phi(\rho_\pm)  \chi (\bdf/\rho_\infty) \big(  \varepsilon+ O(\bdf/\rho_\infty)\big) \Big).
      \eea
     \eeq
   Since $\varepsilon$ is a fixed positive number, it dominates the $O(\rho/\rho_\infty)$ term if we choose $\chi$ supported sufficiently close to $0$. For this reason the value of $\ell|_{L_\pm}$ plays no longer a role.  Then:
    \ben
    \item  In the sink case $L_+$,  the second summand is $\geq 0$, and $>0$ near  $L_{+}\cap \basi$. The first summand is not positive, and supported in $\supp \chi' (\bdf/\rho_\infty)\subset \ff^\inti$. \smallskip
    \item In the source case $L_-$, both summands are $\geq 0$, and $>0$ near $L_-\cap \basi$. 
    \een
    
The  orders in  \emph{Step 2} of the positive commutator argument are slightly modified, namely $A^\12,B \in  \Psiscq^{*,2k,2\ell}(M)$.  Correspondingly, \eqref{eq:cs} is replaced by
$$
\module{\bra  (P-\lambda) u  , Au \ket} \leq \epsilon^{-1}   \|A^\12 (P-\lambda)u\|^2  +  \epsilon \|A^\12 u\|^2.
$$ 
    In  case (1) we eventually obtain an estimate of the form
       $$
        \norm{C_{1} u}_{*,k,\ell+\12} +  \module{\Im \lambda}^{\12}    \norm{C_{1}  u}_{*,k,\ell+\12}    \lesssim   \norm{C_{2} u}_{*,k,*} +  \| A_0(P-\lambda) u \|_{*,k,\ell+\12}+ \errr{u},
            $$
  where $C_1$ has wavefront set near the $\basi$-face, and $\wf'_{\scq}(C_{2})\subset \ff^\inti$ (for this reason the Sobolev norm of $C_2 u$ is of arbitrary order at the $\fibi$ and $\basi$-face). But in the interior of the front face,   Lemma \ref{lem:rad1} applies (case (1)), so taking the $ \norm{C_{2} u}_{*,k,*}$ norm  to be specifically $\norm{C_2 u}_{s,k,\ell+\12}$ with $(k-s)|_{L_+}<-\12$,   $\norm{C_2 u}_{s,k,\ell+\12}$ can be estimated by  the r.h.s~of \eqref{eq:newrad1}. This yields an estimate near the $\basi$-face, and by combining it with the estimate near the $\fibi$-face (case (1) of Lemma \ref{lem:rad1}  again) we obtain \eqref{eq:newrad1} (after relabelling $\ell+\12$ to $\ell$).
  
     In  case (2),  all summands in \eqref{eq:pce} are positive, so we obtain an estimate of the form
       $$
             \norm{C_{1} u}_{*,k,\ell} +  \module{\Im \lambda}^{\12}    \norm{C_{1}  u}_{*,k,\ell}    \lesssim     \| A_0(P-\lambda) u \|_{*,k,\ell}+ \errr{u},
                 $$
  where $C_1$ has wavefront set near the $\basi$-face. By combining it with case (2) of  Lemma \ref{lem:rad1}  we conclude the higher decay estimate \eqref{eq:newrad2}. \qed

    \section{Resolvent estimates}\label{s:resolvent}\init


       \subsection{Fredholm estimates in resolved setting}

       In this section, we move on to the global analysis of $P-\lambda$ for $\Im \lambda >0$.  We make the same assumptions as in the previous sections, i.e.~$P\in \Psi^{2,0}_\sc(M)$ is classical, $
       P-P^*\in  \Psi^{1,-1-\delta}_\sc(M)$ 
       for some $\delta>0$, and $\lambda \in\complex_\varepsilon^+$, where $\complex_\varepsilon^+$ is some subset of $\{ \Im \lambda \geq \varepsilon\}$.
   
   To that end, we make the assumption that $P$ is non-trapping (at zero energy) in the following sense. 
    
   \begin{definition}\label{maindef} We  say that $P$ is \emph{non-trapping}  if there are sinks $L_+\subset \basinf$ and sources $L_-\subset\basinf$ at infinity in the sense of   Definition \ref{sinkssources},  and within $\Char_0$, each bicharacteristic goes either from  $L_+$ to $L_-$, or from $L_-$ to $L_+$, or stays within $L_+$ or $L_-$. We say that a Lorentzian $\sc$-space $(M,g)$ is non-trapping if the wave operator $\square_g$ is non-trapping.
   \end{definition}
   
Under the non-trapping assumption, the lower and higher decay estimates, Propositions \ref{radial3}--\ref{radial4},  can be combined using standard arguments (see e.g.~\cite[\S 5.4]{vasygrenoble} for the wavefront set formulation,  cf.~\cite[\S 3.2]{hassell} for a pedagogical account on the level of estimates) to get the following global estimate). 
   
       \bep[Global estimate in resolved setting]  \label{prop:Fredholm} Suppose $P$ is non-trapping. Let $s,k,\ell\in\cf(\co)$ and assume that $s,k,\ell$ are non-decreasing along the Hamilton flow and  satisfy
           \beq\label{eq:thresh}
           (k-s)|_{L_-}>-\textstyle\12 \mbox{ and } (k-s)|_{L_+}<-\textstyle\12.
           \eeq
        Then for all non-decreasing $s',k',\ell' \in \cf(\co)$  with $k'-s'\in\open{-\12,k-s}$ at $L_-$, for all $u \in\Hsc{s',k',\ell'}$ and all  $\lambda\in\complex_{\varepsilon}^{+}$,
          \beq\label{eq:Fredholm}
          \norm{u}_{s,k,\ell} +  \module{\Im \lambda }^\12    \norm{u}_{s-\12,k,\ell}    \lesssim    \| (P-\lambda) u \|_{s-1,k,\ell}+ \errr{u}.
          \eeq
 
       \eep

      As a conclusion we obtain $L^2$  estimates. Recall that the minimal extension of $P$ is the closure of $P$ acting on $C_{\rm c}^\infty(M)$, and the maximal extension of $P$ (which can be defined as the adjoint of the minimal extension of the formal adjoint $P^*$) has domain $\{ u\in L^2(M) \st Pu \in L^2(M)\}$.

\bet\label{thm:sp} Suppose $P\in \Psi^{2,0}_\sc(M)$ is classical,  $P-P^*\in  \Psi^{1,-1-\delta}_\sc(M)$ for some $\delta>0$, and $P$ is non-trapping in the sense of Definition \ref{maindef}. Then, denoting also by $P$ the minimal extension of $P$, $\sp(P)\setminus \rr$ consists at most of isolated points that are in  $\{ \module{\Im \lambda} \leq R\}$ for some $R>0$. 
\eet 
\begin{proof}  We apply Proposition \ref{prop:Fredholm} with $k=\ell=0$ and  $s\in [0,1]$ such that $s>\12$ on $L_-$ and $s<\12$ on $L_+$. For $\Im \lambda>0$ sufficiently large, the  $\errr{u}$ error in the Fredholm estimate \eqref{eq:Fredholm} can be absorbed into the l.h.s. Furthermore, since $s\geq 0$ and  $s-1\leq 0$, we  conclude  
   $\norm{u}   \lesssim    \| (P-\lambda) u \|$.  The same arguments applied to $-P^*+\bar\lambda$ instead  of $P-\lambda$ give  $\norm{u}   \lesssim    \| (P^*-\bar\lambda) u \|$ for   $\Im \lambda>0$ sufficiently large. It follows that the maximal extension  of  $P-\lambda$   has closed range, and the same is true for the maximal extension of $P^*-\bar{\lambda}$. The two estimates also imply that the minimal and maximal extensions are injective, so it is straightforward to conclude  they are boundedly invertible, in  particular $\lambda \notin \sp (P)$ for large $\Im \lambda$. For general $\Im \lambda>0$, because of the extra $\errr{u}$ error term we obtain the weaker conclusion that the minimal and maximal extension of $P-\lambda$ are Fredholm. By analytic Fredholm theory (note that the reduction to bounded operators is straighforward because the domain of $P-\lambda$ does not depend on $\lambda$), $\sp(P)\cap \{ \Im \lambda >0\}$ consists of isolated points. The case $\Im \lambda<0$ is obtained similarly, with replaced $P$ with $-P$.
\end{proof}

\begin{remark} The arguments from \cite{vasyessential} (not using formal self-adjointness) can be used to prove that there are no accumulation points on the real line apart from possibly $0$.
\end{remark}

\begin{remark} Suppose that $D\in \Psi^{1,0}_\sc(M)$ is such that $P=D^2$ satisfies the hypotheses of \ref{thm:sp}. Then the formula $(D-\lambda)^{-1}:=(D+\lambda)(D^2-\lambda^2)^{-1}$ yield a formal inverse
$$
(D-\lambda)^{-1}: \Hsc{s-1,0}\to\Hsc{s-1,0}
$$
for $s\in [0,1]$ with $s>\12$ on $L_- $ and  $s<\12$ on $L_+$. In particular  $(D-\lambda)^{-1}: L^2_{\rm c}(M)\to L^2_{\rm loc}(M)$.      
\end{remark}
    
   \subsection{Alternative estimate for large \texorpdfstring{$\module{\Im \lambda}$}{imaginary part of lambda}}
   
   We also show a large $\module{\Im \lambda}$ variant of the estimates, with better behaviour in $\Im \lambda$, at the cost of having no gain in regularity.
  
         \bep[Large $\module{\Im \lambda}$ estimate]  \label{prop:Fredholm2var} With the same assumptions as in Proposition \ref{prop:Fredholm}, for sufficiently large $R>0$ we have for all  $u \in\Hscq{s',k',\ell'}$ with  $k'-s'\in\open{-\12,k-s}$ and all $\lambda\in Z_{R}^+$,
            \beq\label{eq:Fredholm2var}
 \module{\Im \lambda }    \norm{u}_{s-\12,k,\ell}    \lesssim    \| (P-\lambda) u \|_{s-\12,k,\ell}.
            \eeq
         \eep
\begin{proof} We  modify the proof of Proposition  \ref{prop:Fredholm} as follows: in \emph{Step 2.}~of the positive commutator argument, when applying the Cauchy--Schwarz inequality to $\module{\bra  (P-\lambda) u  , Au \ket}$ we distribute the norms slightly differently as compared to \eqref{eq:cs}, as to obtain a $\| \cdot \|_{s-\12,k,\ell}$ norm instead of $\| \cdot \|_{s-1,k,\ell}$, and we use Young inequality   with   $\module{\Im \lambda}$ times a constant instead of $\epsilon^2$. This gives the estimate
          \beq\label{eq:Fredholm2}
          \norm{u}_{s,k,\ell} +  \module{\Im \lambda}^\12    \norm{u}_{s-\12,k,\ell}    \lesssim    \module{\Im \lambda}^{-\12} \| (P-\lambda) u \|_{s-\12,k,\ell}+ \errr{u}.
          \eeq
The $ \norm{u}_{s,k,\ell}$ term plays no role, and the $\errr{u}$ term can be absorbed into the l.h.s.~for sufficiently large $\Im\lambda$, so we conclude \eqref{eq:Fredholm2var}.
\end{proof}

For general $\lambda$ with $\Im \lambda >0$, we get that for $s,k,\ell$ non-increasing along the Hamilton flow and $\mp(k-s)|_{L_\pm}>0$, the operator
\beq\label{eq:finv}
P- \lambda : \cX^{s,k,\ell} \to \cY^{s,k,\ell}
\eeq
acting on the Feynman spaces
$\cX^{s,k,\ell}\defeq\{ u\in H_\scq^{s',k',\ell'} \st Pu \in \cY^{s,k,\ell} \}$ and $\cY^{s,k,\ell}\defeq H_\scq^{s,k,\ell}$
is Fredholm, and invertible except for $\lambda$ in a set of isolated poles $I_s\subset \cc$.  Furthermore, $(P-\lambda)^{-1}$ is $O(\bra \Im \lambda \ket^{-1})$ as an operator in $B(\cY^{s,k,\ell},\cX^{s,k,\ell})$.

   \subsection{Resolvent wavefront set} Let us now discuss consequence of the estimates for the Schwartz kernel of the inverse. We are mostly interested in  microlocal properties in the interior $M^\inti$ of $M$, uniformly in $\lambda$. Following \cite{Dang2020}, we use an operator version of the wavefront set.
   
   We write $\Psi_{\rm c}^m(M)$ for the standard class of compactly supported pseudo-differential operators on $M^\inti$. 
   
   \begin{definition}\label{defrrr} Let $Z\subset \cc$ and suppose $\{ G(\lambda)\}_{\lambda\in Z}$  is for all $m\in\rr$ a bounded family of operators in $B(H^m_\c(M),H^m_\loc(M))$.
   The \emph{uniform operator wavefront set of order $s\in\rr$ and weight $\bra \lambda\ket^{-1}$} of $\{ G(\lambda)\}_{\lambda\in Z}$ is the set 
   \beq\label{eq:wfs}
   \wfl{1}\big( G(\lambda) \big)\subset S^* M^\inti  \times S^* M^\inti  
   \eeq
   defined as follows:  $(q_1,q_2)$ is {not} in \eqref{eq:wfs} if and only if for $i=1,2$ there exists $B_i\in \Psi^{0}_{\rm c}(M)$ elliptic at $q_i$ and such that for all $s_0,s_1\in\rr$, $\{ \bra \lambda\ket B_1 G(\lambda) B_2^*\}_{\lambda\in Z}$  \mbox{ is bounded in } $B(H^{s_0}_\c(M), H_\loc^{s_1}(M))$.
   \end{definition}
   
   Let  $\complex_{\varepsilon}^{+}\subset \{ \Im \lambda \geq \varepsilon\}$ for some $\varepsilon>0$ and assume for simplicity  $\complex_{\varepsilon}^{+}\subset \{  \delta \leq \arg \lambda \leq 1-\delta\}$ for some $\delta>0$, so that $\module{\Im \lambda}\sim \bra\lambda\ket$ on $\complex_\varepsilon^+$.

   \begin{proposition}\label{prop:wf} Let  $P$ be as in Theorem \ref{thm:sp}. For $s\in \rr$,  if $\complex_{\varepsilon}^{+}\cap I_s=\emptyset$ then the family  $\{ (P-\lambda)^{-1}\}_{\lambda \in \complex_{\varepsilon}^{+}}$ has Feynman wavefront set in the sense that it satisfies:
   \beq
      \wfl{1}\big( (P-\lambda)^{-1} \big)\subset \{ (q_1,q_2)\in S^* M^\inti  \times S^* M^\inti   \st q_1 \succ q_2 \mbox{ or } q_1=q_2  \}.  
   \eeq
   \end{proposition}
   \begin{proof} Let $(q_1,q_2)$ with $q_1\neq q_2$ and suppose $q_2$ cannot be reached from $q_1$ by a backward bicharacteristic. We can find $B_1,B_2\in \Psi_{\rm c}^0(M)$, formally self-adjoint and elliptic at resp.~$q_1,q_2$, such that every backward bicharacteristic from $\wf'(B_1)$ is disjoint from $\wf'(B_2)$. We want to prove that for all $s_0,s_1\in\rr$ with $s_0\leq s_1$,
   \beq\label{eq:bounded}
   B_1 (P-\lambda)^{-1} B_2 \in B(H^{s_0}_\c(M), H_\loc^{s_1}(M)).
   \eeq
   Let $A_0\in \Psiscq^{0,0,0}(M)$ be elliptic  at the source $L_-$ and  on  the backward bicharacteristics  from $\wf'(B_1)$.   By the higher decay radial estimate, i.e.~Proposition \ref{radial4}, for $s,k,\ell$ non-decreasing along the Hamilton flow and $(k-s)|_{L_-}>-\12$,  
   \beq\label{eq:} 
             \norm{B_1 u}_{s,k,\ell}    \lesssim     \| A_0(P-\lambda) u \|_{s-1,k,\ell}+ \errr{u}.
            \eeq
  In particular, we can take $s$ such that $s\geq s_1$  on $\wf'(B_1)$ and $L_-$, $s-1\leq s_0$ on $\wf'(B_2)$, and $\wf'(A_0)\cap \wf'(B_2)=\emptyset$.  Applying this to $u=(P-\lambda)^{-1} B_2 f$ for $f\in H_\c^{s_0}(M)$  yields
     $$
               \norm{B_1 (P-\lambda)^{-1} B_2 f}_{s,k,\ell}    \lesssim     \| A_0 B_2 f \|_{s-1,k,\ell}+ \errr{(P-\lambda)^{-1} B_2 f}.
               $$
 Since $A_0 B_2\in \Psi^{-\infty}_{\rm c}(M)$ and the last term can be estimated using invertibility of $P-\lambda$ on Feynman spaces \eqref{eq:finv}, taking into account that $B_2$ is compactly supported we conclude 
        $$
                  \norm{B_1 (P-\lambda)^{-1} B_2 f}_{s,k,\ell}    \lesssim     \|  f \|_{s_0},
                  $$
 hence $  \norm{B_1 (P-\lambda)^{-1} B_2 f}_{s_1}    \lesssim   \|  f \|_{s_0}$. This proves  \eqref{eq:bounded}. The $O(\bra \lambda\ket^{-1})$ decay for large ${\Im \lambda}$ is obtained by replacing Proposition \ref{radial4} with the stronger decay version shown in the same way as Proposition \ref{prop:Fredholm2var}. 
 \end{proof}
 
 \subsection{Generalizations for vector bundles and weaker decay} Let us now discuss  generalizations in two possible directions. 
 
 Namely, suppose that we have merely $P-P^*\in  \Psi^{1,-1}_\sc(M)$ instead of $P-P^*\in  \Psi^{1,-1-\delta}_\sc(M)$ with $\delta>0$. Then, the subprincipal symbol $\tilde p$ of $\Im P$ enters the positive commutator estimates, with the effect that the threshold conditions are modified:  \eqref{eq:thresh} is replaced by
  \beq\label{eq:thresh2}
          (k-s)|_{L_-}>-\textstyle\12+\tilde\beta \mbox{ and } (k-s)|_{L_+}<-\textstyle\12+\tilde\beta, 
       \eeq
 where $\tilde\beta\in \cf(\co)$ equals $\tilde p$ times an appropriate weight. For our purposes this is completely unproblematic provided that $|{\tilde \beta}|<\12$, which can be ensured for instance by assuming that $P$ is a small perturbation of  some $P_0$ which satisfies the stronger assumption $\Im P_0\in  \Psi^{1,-1-\delta}_\sc(M)$.
 
 A second, more obvious generalization is to allow for $P$ acting on sections of a smooth vector bundle $E\stackrel{\pi}{\to} M$. We assume that it is endowed with a fiberwise \emph{positive definite} Hermitian form section, which  after integrating on $M$ with the volume form $\vol_g$  defines a scalar product  $\bra \cdot, \cdot \ket_E$. The corresponding Hilbert space is denoted by $L^2(M;E)$. We also assume we are given a scattering connection which gives a notion
 of differentiating sections of $E$ along $\sc$-vector fields, and use it to define $\sc$-pseudo-differential operators $\Psi^{m,\ell}_\sc(M;E)$ and (possibly anisotropic) weighted Sobolev spaces  $H^{m,\ell}_\sc(M;E)$.
 
  Now, if $P\in \Psi^{2,0}_\sc(M;E)$ and $P^*$ stands for the formal adjoint with respect to $\bra \cdot, \cdot \ket_E$, all the arguments in the previous sections apply verbatim.
 
We summarize our results and these observations in the form of a theorem.

 \bet\label{thm:sp2} If $P\in \Psi^{2,0}_\sc(M;E)$ is classical, $P-P^*\in \Psi^{1,-1}_\sc(M;E)$ and there exists $P_0$ such that $P-P_0$ is sufficiently small in $\Psi^{2,0}_\sc(M;E)$, $P_0$ is non-trapping in the sense of Definition \ref{maindef} and  $P_0-P_0^*\in  \Psi^{1,-1-\delta}_\sc(M;E)$ for some $\delta>0$, then the assertions of Theorem \ref{thm:sp} and Proposition \ref{prop:wf}  are true.  
\eet 
 

  \section{Semi-classical estimates}\label{s:semicl}

  \subsection{Resolution of semiclassical scattering cotangent  bundle} In this section we come back to the issue of showing uniform elliptic estimates for $P-\lambda$ and large $\module{\Im \lambda}$ which are needed to finish the proofs of the estimates in the resolved setting in \secs{s:propagation}{s:resolvent}. 
  
  Equivalently, we can constrain $\lambda$ to a fixed bounded region in $\cc \setminus \rr$ and consider as spectral parameter $\lambda/h^2$, and then study the small $h$ regime with semi-classical methods.  The problem is that for the purpose of integration along an infinite contour we need boundedness statements on non-semiclassical spaces (let us call them \emph{classical} spaces for the sake of disambiguation), so the precise relationship between the semiclassical and standard pseudodifferential algebras becomes crucial. For this reason we adopt the second microlocalization point of view proposed recently in \cite{Vasy:Semiclassical-standard} and relying on a blow-up of the corner $h=0$ at fiber infinity,   and we adapt it  to our setting.  While we a priori only need  elliptic estimates to complete the proofs  in \sec{ss:pe2}, we  discuss in one shot the relevant modifications in the non-elliptic regions; this yields more precise propagation estimates and gives a fully microlocal and self-contained proof.

 
We focus on the scalar case to keep the notation short, but note that the generalization to vector bundles is immediate. 
 
  We start with the fiber compactified semiclassical scattering
 cotangent bundle
 $$
 \co \times [0,1]_h,
 $$
 postponing for the moment the corner blow-up at $h=0$.  The associated semi-classical calculus $\Psisch^{m,\ell}(M)$ or $\Psisch^{m,\ell,0}(M)$  is obtained by replacing the quantization map $\Op$ by its semi-classical version $\Op_\semi$, given in the $\rr^n$ case by
 \beq\label{eq:opsemi}
 \Op_\semi(a) u(z)= \frac{1}{(2\pi)^n}\int e^{i (z-z')\cdot \zeta_\semi / h} a(z,\zeta_\semi, h )u(z') dz' d \zeta_\semi.
 \eeq
 Throughout this chapter we use $h$ for
   the actual semiclassical parameter, and $\hbar$ as a subscript to
   denote semiclassical objects, in particular we write $\zeta_\semi$  to stress that the use of the covariable refers to semi-classical quantization. The semiclassical and classical quantization are of course related by the simple change of variables $\zeta_\hbar=h\zeta$, which  is  however singular when we relate symbol classes: in fact, if we wanted to obtain the r.h.s.~of \eqref{eq:opsemi} using the classical quantization $\Op$ then the corresponding symbol is $a(z,\zeta_\semi/h, h )$, which produces problematic $h^{-1}$ powers when differentiated in the second argument.
  
   More generally, operators in $\Psisch^{m,\ell,q}(M)$ are by definition obtained from symbols that come with an extra $h^{-q}$ factor. Leaving aside for now the issue of the relationship with classical pseudodifferential algebras,  the problem with this calculus is that
 even away from the boundary, real principal type estimates are quite 
 degenerate. Indeed, writing $P_{\semi}=h^2 P$ for the
 semiclassical version of $P$, the semiclassical scattering principal symbol of
 the large parameter spectral family,
 $$
 P-\lambda/h^2=h^{-2}(h^2 P-\lambda)=h^{-2}(P_{\semi}-\lambda),
 $$
 is $h^{-2}(p -\lambda)$. 
 Note that 
 $$
 P=h^{-2}P_{\semi}\in\Psisch^{2,0,2}(M),\quad
 \lambda/h^2\in\Psisch^{0,0,2}(M),
 $$
 hence $\lambda$
 has at
 most as high order as the wave operator, which is key for elliptic
 purposes. The issue with this for propagation of singularities
 purposes is that the effective order of $P$ for a commutator argument (which drops
 an order in each sense relative to the elliptic estimate) is $1,-1,1$,
 while the skew adjoint part of the operator (which is a key part of
 propagation estimates), $h^{-2}\Im\lambda$, is still of
 order $0,0,2$, so in the differential sense the wave operator
 dominates, but in every other sense $h^{-2}\Im\lambda$. 
 Correspondingly one
 would like to do some blow ups similarly as in the previous chapter to resolve this; recall that one reason for this
 desire is that then we obtain more precise mapping properties on `standard'
 (rather than function analytically designed) function spaces, but even
 more importantly our problem features radial points at which certain
 threshold inequalities need to be satisfied, and having the resolved
 space ensures that the orders can be chosen {\em naturally}, i.e.\
 without the artificial confines of a pseudodifferential algebra that
 cannot make sufficiently precise distinctions.

 Before proceeding we
 remark that elliptic estimates go through without any problems {\em
   already in the semiclassical scattering algebra}, with
 ellipticity at fiber infinity microlocally given by that of $P$; from the elliptic perspective all further blow ups simply
 allow more general weights, while from the propagation perspective the
 latter are much more important. Also, ellipticity at spatial infinity away from fiber
 infinity, as well as semiclassically in the same sense, is given by
 $\Im\lambda$. 
 
 The only remaining issue is then obtaining estimates at
 fiber infinity near the characteristic set, but in fact even propagation estimates work in the semiclassical scattering algebra already. Namely, by repeating the proofs in \secs{s:propagation}{s:resolvent} one gets estimates of the form
 \begin{equation}\label{eq:prop-sing-complex-abs}
 \|u\|_{s,\ell,q-1}+\|u\|_{s-\12,\ell+\12,q-\12} \lesssim  \|h^{-2}(P_\semi-\lambda)u\|_{s-1,\ell+1,q-2}
 \end{equation}
 with the first term arising from the commutator part and the second
 from $\Im\lambda$,  provided that $\ell>-\12$ at the radial sets where we
 start our estimates propagating and $\ell<-\12$ where the propagation
 ends (and $s,q$ are arbitrary). Note that this is {\em suboptimal} in the orders on the second term.
 
 
 

 \begin{figure}[ht]
\begin{center}
 \includegraphics[width=120mm]{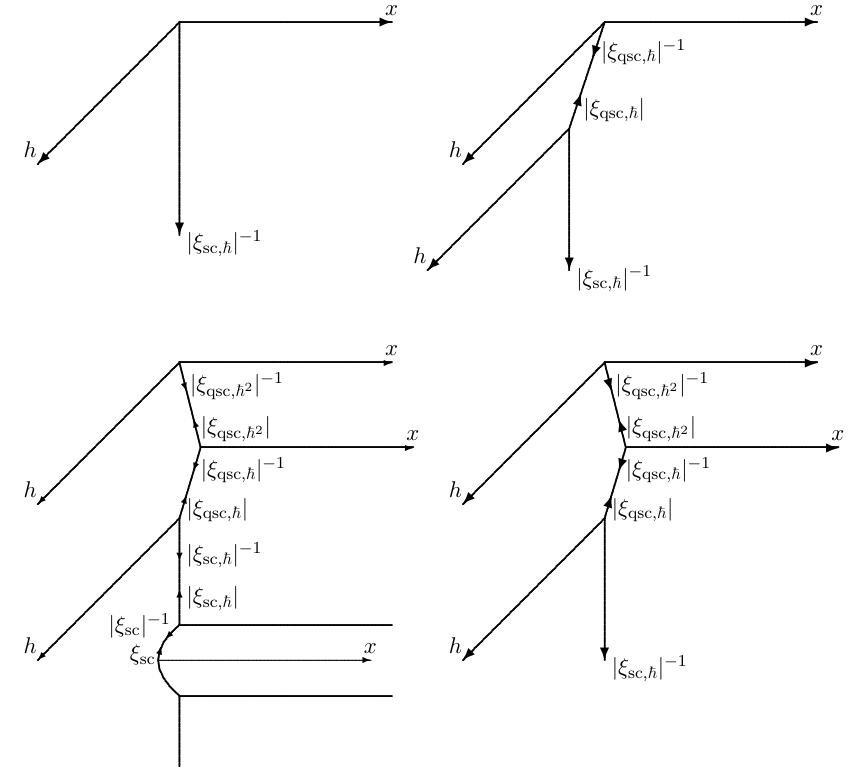}
 \end{center}
 \caption{Resolutions of the fiber compactified semiclassical
   scattering cotangent bundle $\co \times [0,1]_h$. The initial space is on top left. Top
   right is the resolution of fiber infinity at $x=0$, creating the
   semiclassical qsc face. Bottom right is the further resolution of
   $h=0$ at fiber infinity, creating the $h^2$-semiclassical qsc
   face. Finally bottom left adds the blow up of the zero section at
   $h=0$ (which could have been done first); this creates the scattering
   face (with $h$ as a parameter).
 Equivalently one could have started
  with the (non-semiclassical) parameter-dependent fiber compactified scattering cotangent  bundle and blow up fiber infinity at $h=0$---this creates the top left
  picture  away from the semiclassical scattering zero section---and then
  proceed with the next two blowups.
   Geometrically the choices are equivalent; the one we made here
  is a bit simpler. In the second approach the last
  blow up creating the bottom left picture  is not needed, since its front
  face {\em consists} of the fibers of the parameter dependent fiber compactified scattering cotangent
bundle over $h=0$, i.e.~the space we start with in the second approach. } \label{fig:semicl-resolve}
 \end{figure}

 Now moving on to the blow ups, one option is the following: first one
 blows up fiber infinity at $\rho=0$, and then fiber infinity at
 $h=0$; see Figure~\ref{fig:semicl-resolve}. The former simply creates a semiclassical quadratic scattering
 algebra, blown up at the zero section, i.e.\ second microlocalized at
 the zero section. Indeed, this follows from fiber
 infinity at the spatial boundary being defined by
 $\module{\xi_{\scl,\semi}}^{-1}=0$, $\rho=0$, so projective coordinates where
 $\module{\xi_{\scl,\semi}}^{-1}$ is relatively large are
 $\module{\xi_{\scl,\semi}}^{-1}$ and
 $$
 \rho/\module{\xi_{\scl,\semi}}^{-1}=h\rho\module{\xi_{\scl}}=\module{\xi_{\qscl,\semi}},
 $$
 while where $\rho$ is relatively large, the projective coordinates are $\rho$ and
 $\module{\xi_{\scl,\semi}}^{-1}/\rho=\module{\xi_{\qscl,\semi}}^{-1}$. We denote
 the corresponding algebra by $\Psi_{\qscl,\scl,\semi}^{m,k,\ell,q}(M)$ with
 $k$ the new order at the front face, i.e.\ the quadratic scattering
 decay order. In particular,
 corresponding to pullbacks of symbols to the resolved space, we have
 for the non-classical spaces
 \begin{equation}\label{eq:qsc-added-orders}
 \Psisch^{m,\ell,q}(M)=\Psi_{\qscl,\scl,\semi}^{m,m+\ell,\ell,q}(M).
 \end{equation}
 In particular
 $$
 h^{-2}(P_\semi-\lambda)\in\Psi_{\qscl,\scl,\semi}^{2,2,0,2}(M).
 $$
 Exactly as in the previous chapter, our new algebra has the property that symbolic expansions
 come in improvements of size 2 in the new sense, qsc-decay, since they inherit one order gain from both
 boundary hypersurfaces whose intersection was blown up. Also,
 microlocal elliptic estimates from the original $\Psisch^{m,\ell,q}(M)$ calculus are
 inherited, with the advantage that one can also use the new order,
 with elliptic numerology. In this
 algebra one has an improved propagation estimate in two ways.
 
 First, exactly as in the non-semiclassical setting in \sec{ss:pe2}, 
 the radial point estimates have no restrictions at the corner of the
 qsc- and sc-decay faces   since at the sc-decay face $\Im\lambda$
 dominates the commutator, and does so semiclassically as
 well. Correspondingly there are no restrictions on $\ell$. The positive commutator estimate microlocalizes to
 fiber infinity and the condition $\ell>-\12$ or $\ell<-\12$ from the
 scattering perspective translates into, given
 \eqref{eq:qsc-added-orders}, the inequalities  $k-s>-\12$ or
 $k-s<-\12$. This allows both $k$ and $s$ to be high.
 
 Second, in the new space, on the left hand side of
 \eqref{eq:prop-sing-complex-abs} (in the new context), at the corner of the
 qsc- and sc-decay faces the second term, arising from $\Im\lambda$ is
 dominant. Using Cauchy--Schwarz as previously, and absorbing the $\|Au\|^2$ type
 term into the $\Im\lambda$ term, one obtains better numerology, {\em
   with microlocalizers and error term suppressed for the moment, and
   microlocalization only denoted by suppressing the differential order}:
 \begin{equation}\label{eq:prop-sing-complex-abs-qsc-sc}
 \|u\|_{*,k,\ell-\12,q-\12}+\|u\|_{*,k,\ell,q} \lesssim  \|h^{-2}(P_\semi-\lambda)u\|_{*,k,\ell,q-2},
 \end{equation}
 and the first term can simply be dropped.
 Running a similar argument for the corner at the intersection of the
 qsc-decay face and fiber infinity encounters an issue (causing suboptimality) in  that neither term
 dominates at $h=0$: while now they are comparable in the qsc-decay
 sense, and the commutator dominates at fiber infinity, at $h=0$ the
 $\Im\lambda$ term gives the dominant contribution. Proceeding as with
 \eqref{eq:prop-sing-complex-abs} in the current context ({\em thus suboptimally}), regarding the
 commutator term as the dominant for absorbing the $\|Au\|^2$ term in
 the Cauchy--Schwarz inequality, one obtains,  {\em
   with microlocalizers and error term suppressed for the moment, and
   microlocalization only denoted by suppressing the sc-decay order}:
 \begin{equation}\label{eq:prop-sing-complex-abs-qsc-diff}
 \|u\|_{s,k,*,q-1}+\|u\|_{s-\12,k,*,q-\12}\lesssim \|h^{-2}(P_\semi-\lambda)u\|_{s-1,k,*,q-2}.
 \end{equation}
 Note that here we need $k-s>-\12$ or
 $k-s<-\12$, but crucially this allows both $k$ and $s$ to be high. Also
 note that if $h$ is bounded away from $0$ (which is for instance the case for the essential
 self-adjointness discussion), the second term on the left hand side can
 simply be dropped, so that in the non-semiclassical case
 \eqref{eq:prop-sing-complex-abs-qsc-sc} and
 \eqref{eq:prop-sing-complex-abs-qsc-diff} together give an optimal
 result (suppressing the semiclassical order):
 \begin{equation}\label{eq:prop-sing-complex-abs-non-semi}
 \|u\|_{s,k,\ell,*} \lesssim  \|h^{-2}(P_\semi-\lambda)u\|_{s-1,k,\ell,*},
 \end{equation}
 subject to $k-s>-\12$ or
 $k-s<-\12$ in the respective microlocal regions.

 \subsection{Further blow up} In order to resolve this suboptimality, we perform an additional blow up.
 This second
 blow up is then $h=0$ at fiber infinity now defined by
 $|\xi_{\qscl,\semi}|^{-1}=0$, $h=0$, which gives a
 new semiclassical algebra with $h^2$ as the semiclassical parameter. Indeed,
 projective coordinates where
 $|\xi_{\qscl,\semi}|^{-1}$ is relatively large are
 $|\xi_{\qscl,\semi}|^{-1}$ and
 $h/|\xi_{\qscl,\semi}|^{-1}=h^2|\xi_{\qscl}|=|\xi_{\qscl,\semi^2}|$,
 while where $h$ is relatively large, $h$ and
 $|\xi_{\qscl,\semi}|^{-1}/h=|\xi_{\qscl,\semi^2}|^{-1}$. Thus, the new
 face is quadratic scattering, i.e.\ scattering with defining function
 $\rho^2$, with $h^2$ as the effective semiclassical parameter. We denote,
 the new operator space  by $\Psi^{m,k,\ell,p,q}_{\qscl,\scl,\semi^2,\semi}(M)$ with
 $m,\ell,q$ the orders for the lifts of the previous boundary
 hypersurfaces, so $m$ the
 differential order, $\ell$ the sc-decay order and $q$ the
 $h$-semiclassical order, while $p$ is the order corresponding to the
 first blow up, i.e.\ the $h^2$-semiclassical order, and $k$ to the
 second, i.e.\ the quadratic scattering decay order. In particular,
 corresponding to pullbacks of symbols to the resolved space, we have
 for the non-classical spaces
 $$
 \Psisch^{m,\ell,q}(M)=\Psi^{m,m+\ell,\ell,m+q,q}_{\qscl,\scl,\semi^2,\semi}(M),
 $$
 and
 $$
 \Psi_{\qscl,\scl,\semi}^{m,k,\ell,q}(M)=\Psi^{m,k,\ell,m+q,q}_{\qscl,\scl,\semi^2,\semi}(M).
 $$
 In particular
 $$
 h^{-2}(P_\semi-\lambda)\in\Psi^{2,2,0,4,2}_{\qscl,\scl,\semi^2,\semi}(M).
 $$
 Again, our new algebra has the property that symbolic expansions
 come in improvements of size 2 in the two new senses, qsc-decay and
 $h^2$ semiclassical since they inherit one order gain from both
 boundary hypersurfaces whose intersection was blown up.
 
 The advantage of this setup is that at the intersection of the $h$- and
 $h^2$-semiclassical faces, $\Im\lambda$ dominates the commutator, in
 the sense that it has $\geq$ order, with strictly greater order on the
 $h$-semiclassical face, hence it is advantageous to absorb the
 Cauchy--Schwarz error into it and there are no threshold restrictions, while at fiber infinity intersecting
 the $h^2$-semiclassical face the commutator dominates in the analogous
 sense (with strictly greater order at fiber infinity), so it is
 advantageous to absorb errors into it. Correspondingly
 \eqref{eq:prop-sing-complex-abs-qsc-diff} becomes, {\em again
   suppressing the microlocalization},
 \begin{equation}\label{eq:prop-sing-complex-abs-qsc-diff-h2}
 \|u\|_{s,k,*,p,q-1}+\|u\|_{s-1,k,*,p,q} \lesssim  \|h^{-2}(P_\semi-\lambda)u\|_{s-1,k,*,p-2,q-2},
 \end{equation}
 with $k-s>-\12$ or
 $k-s<-\12$ as before; here the first term on the left hand side is the
 near fiber infinity behavior (so the only relevant one there, and thus
 with a corresponding microlocalizer added the second term can be dropped), the
 second is the near $h$-semiclassical behavior, while in the interior
 of the
 $h^2$-semiclassical face the two are comparable.
 
 Altogether we obtain
 \begin{equation}\label{eq:prop-sing-complex-five}
 \|u\|_{s,k,\ell,p,q} \lesssim  \|h^{-2}(P_\semi-\lambda)u\|_{s-1,k,\ell,p-2,q-2},
 \end{equation}
 with $k-s>-\12$ or
 $k-s<-\12$ as before.
 
 \subsection{Resolved semiclassical-classical calculus}Finally, let us place  this in the context of the classical
 scattering algebra following \cite{Vasy:Semiclassical-standard}.
 The classical phase space is introduced into our resolved space by blowing up the
 zero section of the semiclassical cotangent bundle (a neighborhood of
 which has been unaffected by the previous blow ups, so the latter can
 be ignored) at $h=0$ as shown in
 Figure~\ref{fig:semicl-resolve}. Since this is not a corner,
 associating operators with this approach results in serious problems
 due to ill-behaved
 singular symbols. As mentioned already in the caption below  Figure~\ref{fig:semicl-resolve} it is better to
 instead introduce the semiclassical phase space into  the classical
 one by blowing up fiber infinity in the
 parameter-dependent fiber compactified scattering cotangent bundle at
 $h=0$. 
  Indeed, this
 approach has the advantage that both this blow up as well as all
 subsequent ones performed above in the semiclassical setting resolve corners and thus all analytic constructions
 are essentially unaffected (though one obtains more general orders,
 etc). Incorporating this into the operator algebra produces yet
 another order, which we append to the end of our long list $\Psi_{\qscl,\scl,\semi^2,\semi,\mathrm{cl}}^{m,k,\ell,p,q,r}(M)$.

While ellipticity in $\Psi_{\qscl,\scl,\semi^2,\semi,\mathrm{cl}}^{m,k,\ell,p,q,r}(M)$ is defined as in previous constructions, full ellipticity defined as in  \cite{Vasy:Semiclassical-standard} requires additionally invertibility of the normal operator  at the `classical' or `parameter' hypersurface (corresponding to the
 last order) $\rho_{\cl}=0$, namely:

\begin{definition} We say that $A\in \Psi_{\qscl,\scl,\semi^2,\semi,\mathrm{cl}}^{m,k,\ell,p,q,r}(M)$ is \emph{fully elliptic} at  $\rho_{\cl}=0$ if it is elliptic and there exists an invertible operator $A_0\in \Psi_{\qscl,\scl,\semi^2,\semi,\mathrm{cl}}^{\infty,\infty,\infty,p,q,r}(M)$ with 
 $$
A_0^{-1}\in \Psi_{\qscl,\scl,\semi^2,\semi,\mathrm{cl}}^{\infty,\infty,\infty,-p,-q,-r}(M)
 $$ 
 and such that $A-A_0\in \Psi_{\qscl,\scl,\semi^2,\semi,\mathrm{cl}}^{\infty,\infty,\infty,p-1,q-1,r}(M)$.
\end{definition}
 
 Full ellipticity of $A$ implies that the elliptic parametrix can be improved through a standard iterative construction to a parametrix modulo $ \Psi_{\qscl,\scl,\semi^2,\semi,\mathrm{cl}}^{-\infty,-\infty,-\infty,-\infty,-\infty,-\infty}(M)$, which in particular implies invertibility for small $h$. The notion of full ellipticity can be extended to situations where $A$ is elliptic only microlocally in some region, and then the same construction gives a microlocal parametrix improved  in the last three orders.
 
 Proceeding as previously, in the new calculus we derive estimates for
 $$
 P-\lambda/h^2=h^{-2}(P_\semi-\lambda)\in\Psi_{\qscl,\scl,\semi^2,\semi,\mathrm{cl}}^{2,2,0,4,2,2}(M).
 $$
Keeping in mind (microlocal) full ellipticity at $\rho_{\cl}=0$ (the normal operator is simply $h^{-2}\lambda$ times the identity, which is invertible), we obtain eventually the following result.
 
\begin{theorem} With the same notation and hypotheses  as in  Proposition \ref{prop:Fredholm}, in particular assuming the threshold conditions $(k-s)|_{L_-}>-\textstyle\12$ and $(k-s)|_{L_+}<-\textstyle\12$, for all $p,q,r\in \rr$,  all $\lambda$ in a compact subset of $Z_\varepsilon^+$ and  $h$ sufficiently small, we have 
\begin{equation}\label{eq:prop-sing-complex-six}
 \|u\|_{s,k,\ell,p,q,r}\lesssim \|(P-\lambda/h^2)u\|_{s-1,k,\ell,p-2,q-2,r-2}.
 \end{equation}
 \end{theorem}

Estimates on classical spaces can be now deduced similarly as in \cite{Vasy:Semiclassical-standard}. Namely, let us denote by $\Psi_{\sc,{\rm cl}}^{m,\ell,t}(M)$ the calculus obtained through applying non-semiclassical quantization to symbols with $O(h^{-t})$ seminorms, and by $\Psi_{{\rm qsc},\sc,{\rm cl}}^{m,k,\ell,t}(M)$ the corresponding quadratic scattering-scattering calculus. Then
$$
\Psi_{{\rm qsc},\sc,{\rm cl}}^{m,k,\ell,t}(M)= \Psi_{\qscl,\scl,\semi^2,\semi,\mathrm{cl}}^{m,k,\ell,m+t,t,m+t}(M)$$
and
$$ \Psi_{\sc,{\rm cl}}^{m,\ell,t}(M)= \Psi_{\qscl,\scl,\semi^2,\semi,\mathrm{cl}}^{m,m+\ell,\ell,m+t,t,m+t}(M)
$$
using the relationships between the boundary defining functions (this is analogous to formula (1.2) in \cite{Vasy:Semiclassical-standard}). 
Thus, with the  obvious notation for corresponding variable order Sobolev spaces and keeping in mind the threshold conditions $(k-s)|_{L_-}>-\textstyle\12$ and $(k-s)|_{L_+}<-\textstyle\12$, the estimate \eqref{eq:prop-sing-complex-six} translates to uniform boundedness of
$$
(P-\lambda/h^2)^{-1}:  H_{{\rm qsc},\sc,{\rm cl}}^{s-1,k,\ell,t}(M)=H_{\qscl,\scl,\semi^2,\semi,\mathrm{cl}}^{s-1,k,\ell,s+t-1,t,s+t-1}(M)\to H_{\qscl,\scl,\semi^2,\semi,\mathrm{cl}}^{s,k,\ell,s+t+1,t+2,s+t+1}(M),
$$
with the latter space continuously embedded in
$$
H_{\qscl,\scl,\semi^2,\semi,\mathrm{cl}}^{s-\theta,k,\ell,s+t+1,t+1+\theta,s+t+1}(M)=H_{{\rm qsc},\sc,{\rm cl}}^{s-\theta,k,\ell,t+1+\theta}(M)
$$
for $\theta\in [0,1]$.  In particular, the choice $\theta=\12$ corresponds to an improvement of the high energy estimate stated in Proposition \ref{prop:Fredholm2var}. 

Similarly we obtain the (less useful for us) uniform boundedness of
$$
(P-\lambda/h^2)^{-1}: H_{\sc,{\rm cl}}^{s-1,\ell+1,t}(M) \to H_{\sc,{\rm cl}}^{s-\theta,\ell+\theta,t+1+\theta}(M)
$$
for $\theta\in [0,1]$ if $\ell|_{L_-}>-\textstyle\12$ and $\ell|_{L_+}<-\textstyle\12$.
   
\section{Dirac operators} \label{s:dirac}

\def\SM{{S}}
\def\dual{\!\cdot \!}
\def\varn{\mathbf{n}}

\subsection{Dirac operators on scattering bundles}  

Let $(M,g)$ be a Lorentzian sc-space. Then the Levi-Civita connection $\nabla$ on $(M,g)$ is a \emph{scattering connection}, in particular
$$
\nabla : \cf(M)\to \cf(M ; \be T^*M),
$$
 and it is the lift of a $\b$-connection (see \cite[\S1.5]{Kottke2015}), in particular $\nabla_X: \cf(M) \to \bdf\cf(M)$
for all $X\in \cV_\sc(M)$.

 Various definitions from spin geometry can be adapted to the sc-setting, essentially  by replacing $TM$ by $\be TM$ and $T^*M$ by $\be T^*M$, see for instance \cite{Kottke2015} (cf.~\cite{Melrose1993} for the analogous constructions in the $\b$-setting). In particular, one can define a complex \emph{scattering spinor bundle} $\SM \to M$  in analogy to complex spinor bundles. It is endowed with a linear map
$$
\gamma: \cinf(\M; \be T\M)\to \cinf(\M; \End(\SM))
$$ called \emph{Clifford multiplication}, satisfying the Clifford relations
\beq\label{clifford}
\gamma(X)\gamma(Y)+ \gamma(Y)\gamma(X)= -2 X\dual  g Y \one, \ \ X, Y\in \cinf(\M; \be T\M),
\eeq
 and such that for each $x\in M$, $\gamma_{x}$ induces a faithful irreducible representation of the Clifford algebra ${\rm  Cl}(\be T_{x}M, g_{x})$ in $S_{x}$. Above, $g$ is understood as a section of ${\rm Hom}(\be TM,\be T^*M)$, and $\,\cdot\,$ is the fiberwise duality pairing (not to be confused with Clifford multiplication). 
  
Rather than discussing  spin structures in detail {(we refer to \cite[p.~548--551]{BGM}, \cite[p.~12--13]{Damaschke}, \cite[p.~3--5]{VanDenDungen2018} for more background)}, we assume that we are given a connection $\nabla^{S}$ on $\SM$, called \emph{spin connection}, such that for all $X, Y\in \cinf(\M; \be T\M)$ and  $u\in \cinf(\M;\SM)$,
$$
\nabla_{X}^{S}(\gamma(Y)u)= \gamma(\nabla_{X}Y)u+ \gamma(Y)\nabla_{X}^{S}u.
$$
  Furthermore, we assume that there is a section $\beta\in \cinf(M; {\rm Hom}(S,S^*))$ such that $\beta_{x}$ is Hermitian non-degenerate for each $x\in M$, and: \smallskip
\ben 
\item[a)]  $\gamma(X)^{*}\beta = - \beta \gamma(X)$ for all $X\in \cV_{\rm sc}(M)$;\smallskip
\item[b)]   $i \beta \gamma(e)>0$ for  $e\in \cV_{\rm sc}(M)$ time-like and future directed on  $M^\inti$;\smallskip
\item[c)] $X(\overline{u}\dual \beta v)= \overline{\nabla_{X}^{S}u}\dual \beta v+ \overline{u}\dual \beta \nabla_{X}^{S}v$ for all $X\in \cV_{\rm sc}(M)$ and $u,v\in \cinf(\M;\SM)$.\smallskip
\een

In $\rm b)$, it suffices that the condition $i \beta \gamma(e)>0$ is satisfied for some vector field $e\in \cV_{\rm sc}(M)$ which is time-like and future directed on  $M^\inti$, and then it is satisfied for \emph{all} such vector fields, see \cite[p.~13]{Damaschke}, \cite[p.~4-5]{VanDenDungen2018}.

By integrating $u,v \mapsto u \dual \beta v$ with the  volume form on $(M,g)$ we obtain a Hermitian form $\bra \cdot,\cdot \ket_{\SM}$ on sections of $\SM$. We stress   that  $\bra \cdot,\cdot \ket_{\SM}$ is {not} positive definite in Lorentzian signature.  

In this setting, the  \emph{Dirac operator} is the scattering differential operator given in a local orthonormal frame  $(e_{0}, e_1,\dots, e_{n-1})$ of $\be T\M$ by
\beq\label{eq:Dirac}
{\D}=g^{\mu\nu}\gamma(e_{\mu})\nabla^{S}_{e_{\nu}} \in {\rm Diff}^1_{\sc}(M;S),
\eeq
where we sum over repeated indices and $g^{\mu\nu}$ is the inverse metric (understood as a section of $\be TM\otimes_{\rm s} \be TM$). Using properties $\rm a)$ and $\rm c)$ above and Stokes' theorem one can show that $\D$ is formally self-adjoint for the non-positive  Hermitian form $\bra \cdot,\cdot \ket_{\SM}$. 

To work in a Hilbert space setting, we use $\rm b)$ to define a scalar product. Namely, we fix some vector field $e$ with the properties given in $\rm b)$, and we set
\beq\label{eq:scal}
\bra u  ,  v\ket \defeq  \bra u ,  \gamma(e)  v \ket_S,
\eeq
and write $L^2(M;S)$ for the corresponding Hilbert space. 

\begin{remark}\label{rem:can} If  $(M^\inti,g)$ is  globally hyperbolic then there is a canonical choice consisting of taking $e$ to be the unique past-directed vector field such that
$$
e(x)\dual g(x) e(x)=-1 \mbox{ and } e(x) \dual g(x) v=0
$$
for all $x\in M^\inti$ and all $v$ that are tangent to the foliation by Cauchy surfaces. The corresponding scalar product $\bra \cdot , \cdot \ket$ is then the scalar product arising in field quantization. Note that here in addition we need that $e\in \cV_{\rm sc}(M)$, which requires some extra mild assumptions. 
\end{remark}

\begin{lemma}\label{lem:Dsquare} Suppose $\D$ is a Dirac operator on a scattering spinor bundle over a Lorentzian scattering space $(M,g)$. If $(M^\inti,g)$ is globally hyperbolic then $P\defeq -\D^2$ satisfies 
 $P-P^*\in  \Psi^{1,-1}_\sc(M;\SM)$, 
 where $P^*$ is the formal adjoint for to the scalar product $\bra  \cdot ,  \cdot \ket$ defined in \eqref{eq:scal}.
\end{lemma}
\begin{proof} Without loss of generality we can assume that $e_0=e$ and $e\dual g e=-1$. The Clifford relations imply $\gamma(e)^2=-1$, and since $\D$ is formally self-adjoint for $\bra \cdot , \cdot \ket_S$, we get $(\D^2)^*=-\gamma(e) \D^2 \gamma(e)$ for $\bra \cdot , \cdot \ket$. By commuting $\gamma(e)$ through $\D^2$, using Clifford relations we find that $(\D^2)^*=\D^2+R$, where $R$ is the sum of terms in $\Psi^{1,0}_\sc$ times either  $[\nabla^S_{e_0},\nabla^S_{e_\mu}]\in \Psi^{0,-1}_\sc$ or $[\nabla^S_{e_\mu}, \gamma(e_0)]\in\Psi^{0,-1}_\sc$, hence  $R\in\Psi^{1,-1}_\sc$.
\end{proof}

Recall that on top of a non-trapping hypothesis our estimates require  that the imaginary part  $\12(P-P^*)$ should either have better decay and be in 
$\Psi^{1,-1-\delta}_\sc(M;\SM)$ for some $\delta>0$   (Theorem \ref{thm:sp2}), or be small in $\Psi^{1,-1}_\sc(M;\SM)$ semi-norms (see Theorem \ref{thm:sp2}).

This motivates the following definition, which we formulate in the more general setting of a Clifford module $E\to M$. Recall that Clifford modules are defined similarly as spin bundles except that we no longer require that $\gamma_x$ is the Spin representation of the Clifford algebra, and each fiber ${E}_x$ is a left module for the Clifford algebra ${\rm Cl}(T_xM)\otimes_{\mathbb{R}}\mathbb{C} $ (see \cite[p.~42]{Roe}) for the parallel discussion in the Riemannian case). As explained therein, a typical example of Clifford module can be obtained by tensoring the spin bundle $S$ by an auxiliary Hermitian vector bundle  
$(H,\nabla^H)$ with vertical metric $h$ and endowed with a Hermitian connection $\nabla^H$, then the resulting 
bundle ${E}=S\otimes H$ endowed with Clifford connection $\nabla^{{E}}=\nabla^S\otimes \one_H+\one_S\otimes \nabla^H$ is indeed a Clifford module.
The Clifford action $\gamma$ on $S$ gets induced on $S\otimes H$ just by acting on the first factor of the tensor product. 
The Hermitian form $\beta$ on spinors in ${S} $ gets upgraded on the twisted module ${S}\otimes H$ by setting
$$ \overline{(s\otimes f)} (\beta\otimes h) (s\otimes f):= \left(\overline{s} \beta (s) \right) \left( \overline{f}hf  \right)  $$
for a section $s\otimes f$ of $S\otimes H $ and extended by bilinearity,
hence the positivity property of $i(\beta\otimes h)\gamma(e)$ is preserved.

\begin{definition}\label{def:SM} We will say that $\D\in{\rm Diff}^1_{\sc}(M;E)$ is a \emph{Dirac operator on a small perturbation of (generalized) Minkowski space}  if $\D$ is given by formula \eqref{eq:Dirac}, possibly with $S\to M$ replaced by a Clifford module $E\to M$, and   $P=-\D^2$ satisfies either
\begin{enumerate}
\item  $P-P^*\in\Psi^{1,-1-\delta}_\sc(M;\SM)$ for some $\delta>0$ and is non-trapping in the sense of Definition \ref{maindef}, or,
\item $P$ is close to a non--trapping operator $P_0\in\Psi^{2,0}_\sc(M;E)$ as above, i.e.~$P-P_0$ is sufficiently small in $\Psi^{2,0}_\sc(M;E)$ and $P_0-P_0^*\in  \Psi^{1,-1-\delta}_\sc(M;E)$ for some $\delta>0$.
\end{enumerate}
\end{definition}

The term \emph{generalized} appears in the definition because we allow a priori for more general reference geometries than Minkowski space. However, we drop that term in the sequel as  the most natural examples are provided by perturbations of Minkowski space.

Indeed, version (2)  is verified for instance in the case when $M=\overline{\rr^n}$ and $g$ is a sufficiently small sized perturbation of the Minkowski metric (in the sense of scattering metrics) and $S\to M$ is the spinor bundle canonically obtained from the unique spin structure on $(M,g)$ (existence and uniqueness is a consequence of $\rr^n$ being parallelizable and orientable; one can then examine the construction to conclude that it yields a scattering bundle). Version (1) of the definition is verified when the perturbation is sufficiently fastly decaying; we give an elementary example below.

\begin{lemm}
Let $M=\overline{\rr^n}$, $g_0$ the flat Minkowski metric and $g$ a scattering metric such that $g-g_0$ is Schwartz. Let $\D$ be the corresponding Dirac operator and $P=-\D^2$. Then  
$P-P^*\in \Psi^{1,-\infty}_\sc(M;S)$. 
\end{lemm} 
\begin{proof}
  The role of $P_0$ is played by minus the square of the Minkowski Dirac operator, which satisfies $P_0^*=P_0$ w.r.t.~the scalar product defined using the vector field $e_0$ from Remark \ref{rem:can}. All commutators appearing in the proof of Lemma \ref{lem:Dsquare} have fast decay   since $g-g_0$ is Schwartz. More precisely, the difference
$P-P^*$ can be controlled since
$$
\bea P\gamma(e)-\gamma(e)P&=P(\gamma(e)-\gamma_0(e_0))+(P-P_0)\gamma_0(e_0)\fantom+P_0\gamma_0(e_0)-\gamma_0(e_0)P_0+\gamma_0(e_0)(P_0-P)+(\gamma_0(e_0)-\gamma(e))P, 
\eea
$$
which is in $\dot\cC$  
because $P-P_0\in \Psi^{2,-\infty}_\sc(M;S)$, $\gamma(e)-\gamma_0(e_0)$ has fast decay, and $P_0\gamma_0(e_0)=0$ since $P_0$ has constant coefficients and $\gamma_0(e_0)$ is a constant section.
\end{proof}

\subsection{Bochner--Lichnerowicz formula} For the purpose of relating functions of $P$ with e.g.~the scalar curvature, a crucial tool is the Bochner--Lichnerowicz formula.
We briefly introduce its ingredients; note that here the discussion merely requires that ${E}\to M$ is a Clifford module on an a priori general Lorentzian manifold $(M,g)$. 

Let $K\in \Omega^2_\sc(M;\text{End}({E}))$ denote the curvature of
the connection $\nabla^{{E}}$ on ${E}$ at $x\in M$. In an orthonormal frame $(e_i)_i$ with vanishing Lie bracket at $x$, $K(e_i,e_j)=[\nabla^{{E}}_{e_i},\nabla^{{E}}_{e_j}]\in C^\infty_{\sc}(\text{End}({E}))  $. 
We introduce the corresponding Clifford contraction $\mathbf{K}\in C^\infty_{\sc}(\text{End}({E}))$ of $K$ defined as
$\mathbf{K}=\sum_{i<j} \gamma(e_i)\gamma(e_j) K(e_i,e_j) $.

Then, following the proof of \cite[p.~44]{Roe} (which applies verbatim to the Lorentzian case) we find that 
$$
\D^2=\nabla^{{E}*}\nabla^E+\mathbf{K}
$$ with $\nabla^{E*}$ being the formal adjoint of $\nabla^E$ for the non-positive inner product obtained from $\beta$.
If $R^{TM}$ is the Riemann curvature tensor, we define the Riemannian endomorphism as $R^{{E}}\in \Omega^2_\sc(M;\text{End}({E}))$ as the  operator
$$R^{{E}}(X,Y)=\sum \gamma(e_k)\gamma(e_l) \left\langle R(X,Y)e_k,e_l\right\rangle_g.$$
By a similar proof as in \cite[p.~46--48]{Roe} we find that
the curvature operator $K$ (given by $K(e_i,e_j)=[\nabla^{{E}}_{e_i},\nabla^{{E}}_{e_j}]$) decomposes as
$$
K=R^{{E}}+F^{{E}},
$$
where $F^{{E}}$ is an element of $\Omega^2(M,\text{End}({E}))$ {commuting} with the Clifford action and called the \emph{twisting curvature}. 


The Bochner--Lichnerowicz formula  says that
\beq\label{eq:BL}
\slashed{D}^2=\nabla^{{E}*}\nabla^{{E}}+\mathbf{F}^{{E}}+\frac{1}{4}R_g\one_{{E}},
\eeq
where the formal adjoint is taken \emph{w.r.t.~the non-positive Hermitian form} $\bra\cdot,\cdot \ket_S$ (in contrast to other formulas in the paper), $R_g\in \cf(M)$ is the scalar curvature of $(M,g)$ and $$
\mathbf{F}^{{E}}=\sum_{i<j}\gamma(e_i)\gamma(e_j)F^{{E}}(e_i,e_j)\in \cf(M;\End({E}))
$$ is the Clifford contraction of the twisting curvature. The proof in the Riemannian case applies verbatim, see e.g.~\cite[Prop.~3.18 p.~48]{Roe},~\cite[Thm.~3.52]{BGV} and \cite[Thm.~8.8]{lawsonspin}.

We will also make crucial use of the formula (with sum over repeated indices) 
\begin{eqnarray*}
\nabla^{{E}*}\nabla^{{E}}  =  
-\nabla^{{E}}_j g^{ij}\nabla^{{E}}_i-g^{ij}{\vert g\vert}^{-\12} (\partial_j{\vert g\vert}^{\12}) \nabla^{{E}}_i,
\end{eqnarray*}
which in combination with the Bochner--Lichnerowicz identity \eqref{eq:BL} gives 
\beq\label{eq:BL2}
\slashed{D}^2=-\nabla^{{E}}_jg^{ij}\nabla^{{E}}_i-g^{ij}{\vert g\vert}^{-\12} (\partial_j{\vert g\vert}^{\12})\nabla^{{E}}_i+  \mathbf{F}^{{E}}+\frac{1}{4}R_g\one_{E}.
\eeq
The reader can find in \cite[\S11.3 p.~206]{Connes2008a} an analogue discussion of the above decomposition in the Riemannian case.

\subsection{Complex powers and local invariants} Although $P$ is not necessarily self-adjoint, we can still define its complex powers at follows.

\begin{definition}\label{def:cp} For $\varepsilon>0$, let  $\tilde Z_\varepsilon\subset \cc$ be the following contour in the upper half-plane
$$
\tilde Z_{\varepsilon} = e^{i(\pi-\theta)}\opencl{-\infty,\textstyle\frac{\varepsilon}{2}}\cup \{\textstyle\frac{\varepsilon}{2} e^{i\omega}\, | \, \pi-\theta<\omega<\theta\}\cup  e^{i\theta}\clopen{\textstyle\frac{\varepsilon}{2},+\infty}.
$$  
Let $Z_{\varepsilon}$ a small deformation of $Z_{\varepsilon}:=\tilde Z_{\varepsilon}+i\varepsilon$ which does not intersect $\sp(P)$.   For $\Re \cv > 0$, we define the operator
\beq\label{eq3}
(P-i\varepsilon)^{-\cv}:=\frac{1}{2\pi i}\int_{Z_\varepsilon} (z-i\varepsilon)^{-\cv} (P-z)^{-1} dz,
\eeq
where the integral converges in the strong operator topology of $B(L^2(M;E))$ since  $\sp(P)\subset \{ \module{ \Im \lambda }< R  \}$.
\end{definition}

Note that the definition depends on the choice of contour $Z_\varepsilon$: one should think of $(P-i\varepsilon)^{-\cv}$ as being defined canonically \emph{up to spectral subspaces associated with possible complex poles}. We  prove that for two choices of contours $Z_{1,\varepsilon}, Z_{2,\varepsilon}$ with the same asymptotics at infinity, the complex powers differ at most by a finite rank operator.

\begin{lemma}\label{l:ambiguity_complex_powers}
Let $Z_{1,\varepsilon}$ be a contour as in Definition \ref{def:cp} and let
$Z_{2,\varepsilon}$ be any other contour in the upper half-plane avoiding $\sp(P)$ and obtained by a  finite deformation of $Z_{1,\varepsilon}$. Then the difference 
\begin{eqnarray*}
\frac{1}{2\pi i}\left(\int_{Z_{1,\varepsilon}} (z-i\varepsilon)^{-\cv} (P-z)^{-1} dz-\int_{Z_{2,\varepsilon}} (z-i\varepsilon)^{-\cv} (P-z)^{-1} dz\right)
\end{eqnarray*}
is an operator of finite rank
whose Schwartz kernel is $C^\infty$ in $M^\inti$.
\end{lemma}
\begin{proof}
First, the difference of $1$-chains $Z_{1,\varepsilon}-Z_{2,\varepsilon}$ is a compactly supported boundary enclosing a compact domain $\mathbb{D}\subset \{z \in \mathbb{C} \st  \Im z >0\}$. Therefore, by analytic Fredholm theory, if we denote by $\text{Res}(P)$ the resonances of $P$ in the upper half-plane then
$$
\frac{1}{2\pi i}\int_{\partial\mathbb{D}} (z-i\varepsilon)^{-\cv} (P-z)^{-1} dz=\sum_{\lambda\in \text{Res}(P)\cap \mathbb{D}  } (\lambda-i)^{-\cv} \Pi_\lambda, 
$$
where $\Pi_\lambda$ is the spectral projector on the resonant states of $P$ for $\lambda\in \text{Res}(P)$; here we made use of the fact that analytic Fredholm theory ensures the poles of the meromorphic resolvent to have finite multiplicity.
The resonance eigenspace $E_\lambda=\Ran \Pi_\lambda$ is not necessarily spanned by
eigenfunctions of $P$ as there might be Jordan blocks. It suffices in that case to discuss the regularity of cyclic families within 
$E_\lambda$, meaning some $N$-uple of resonant states $(u_1,\dots,u_N)$ of $E_\lambda$ such that $(P-\lambda)u_1=0$,  $(P-\lambda)u_2=u_1$, $\dots$,  $(P-\lambda)u_N=u_{N-1}$.
We  know that $u_1\in H^{s,k,\ell}_\scq(M)$ for some  $s,k,\ell \in C^\infty(  \overline{ ^{\rm sc}T^*}M )$ which are non-decreasing along the flow and such that $(k-s)|_{L_-}>-\frac{1}{2} $ and $(k-s)|_{L_+}<-\frac{1}{2}$.
By the Fredholm estimate from Proposition \ref{prop:Fredholm}, using $(P-\lambda)u_1=0$ we find that
$$ \Vert u_1\Vert_{s+1,k,\ell}+\module{\Im\lambda}^{\frac{1}{2}}\Vert u_1 \Vert_{s+\frac{1}{2},k,\ell}  \lesssim \Vert u_1\Vert_{\scriptscriptstyle{S},\scriptscriptstyle{K},\scriptscriptstyle{L}} .   $$
The above estimate tells us that $u$ belongs in fact to $H_{\rm sc,qsc}^{s+1,k,\ell}$. 
Now letting $s$ become $s+1$ and $k$ become $k+1$ and bootstrapping the above estimate yields $u_1\in C^\infty$ in $M^\inti$, more precisely $u_1\in \bigcap_{p\geqslant 0} H^{s+p,k+p,\ell_p}_{\rm sc,qsc}$ where $\ell_p$ is some sequence depending on $p$ which is irrelevant as it concerns only the decay at infinity. Now using the equation $(P-\lambda)u_2=u_1$ and again the global estimate for $u_2$ yields
$$ \Vert u_2\Vert_{s+1+p,k+p,\ell_p}+\module{\Im \lambda}^{\frac{1}{2}}\Vert u_2 \Vert_{s+\frac{1}{2}+p,k+p,\ell_p}  \lesssim \Vert u_1 \Vert_{s+p,k+p,\ell_p} + \Vert u_2\Vert_{\scriptscriptstyle{S}_2,\scriptscriptstyle{K}_2,\scriptscriptstyle{L}_2}    $$
for all $p$, hence  the regularity of $u_2$. Then using the cyclic structure, we can propagate the regularity to $u_1,\dots,u_N$.
 Repeating now the same proof for coresonant states, i.e.~generalized resonant states of the adjoint $P^*$ for the resonance $\lambda$, we deduce that 
both $\Ran \Pi_\lambda$ and $\Ran \Pi_\lambda^*$ have images contained in $C^\infty(M)$. This yields the smoothness of the kernel of $\Pi_\lambda$ by duality and concludes the proof. 
\end{proof}

The dependence on the contour of integration   turns out to be inessential in the  formula that we  prove, which gives a  relationship between complex powers of $P$ and the local geometry of $(M,g)$. We deduce it   using a Hadamard parametrix argument similarly as in as in \cite{Dang2020}, with  necessary adjustments in the vector bundle, non self-adjoint case.

\begin{theorem}\label{thm:final} Let $(M,g)$ be a Lorentzian scattering space of even dimension $n$. Suppose $\D \in{\rm Diff}^1_{\sc}(M;E)$ is a Dirac operator on a small perturbation of Minkowski in the sense of  Definition \ref{def:SM}, and let  $P=-\D^2$.
 Then for all $\varepsilon>0$, the Schwartz kernel of $(P-i \varepsilon)^{-\cv}$ has for $\Re\cv>\frac{n}{2}$ a well-defined on-diagonal restriction $(P-i \varepsilon)^{-\cv}(x,x)$, which extends as a meromorphic function of $\cv\in\cc$ with poles at $\{\n2,$ $\n2-1$, $\n2-2$, $\dots$, ${1}\}$. Furthermore, 
$$ 
\bea 
\lim_{\varepsilon\rightarrow 0^+} \res_{\cv=\frac{n}{2}-1} \tr_E\left((P- i\varepsilon)^{-\cv}\right)(x,x)= \frac{{\rk}(E)  R_g(x)  }{{i}6(4\pi)^{\n2} \Gamma\big(\frac{n}{2}-1\big) } +  \frac{  2 \tr_E \big( F^E\big)(x)  }{{i}(4\pi)^{\n2} \Gamma\big(\frac{n}{2}-1\big) } ,
\eea 
$$
where $\rk(E)$ is the rank of $E\to M$.
\end{theorem}
\begin{proof} With the ingredients from the previous chapters at hand, the proof is largely  analogous to the scalar case \cite[Thm.~1.1]{Dang2020}, with minor adaptations needed to deal with  bundle aspects and possible non-real spectrum. Let us recall the main steps of the proof in the scalar case and explain how they are modified:

\step{1}  It is shown there that for $s\in \rr$ and for $\lambda$ along $Z_\varepsilon$, the uniform wavefront set       $\wfl{12}\big( (P-\lambda)^{-1} \big)$ is contained in the Feynman wavefront. In the present setting, this is replaced by the more general Proposition \ref{prop:wf}, which applies to $P=-\D^2$ thanks to Lemma \ref{lem:Dsquare} and keeping in mind that $Z_\varepsilon$ does not intersect $\sp(P)$.  
Note also that Proposition \ref{prop:wf} yields even better, $O(\bra\lambda \ket^{-1})$ decay along $Z_\varepsilon$.  

\step{2} Next, \cite{Dang2020} constructs a \emph{Hadamard parametrix} $H_N$ for $(P-\lambda)^{-1}$. First, one defines a family of distributions on $\rr^n$,  holomorphic in $\cv\in\cc\setminus\{-1,-2,\dots\}$ for $\Im \lambda\geqslant 0$, $\lambda\neq 0$, given by the Fourier transform
$$
\Fs{\lambda}\defeq\frac{\Gamma(\cv+1)}{(2\pi)^{n}} \int e^{i\left\langle x,\xi \right\rangle}\left(\vert\xi\vert_\eta^2-i0-\lambda\right)^{-\cv-1}d^{n}\xi,
$$
where $\vert\xi\vert_\eta^2 =  -\xi_0^2+\sum_{i}\xi_i^2$ is minus the Minkowski quadratic form. The distributions $\Fse{z}$ can be pull-backed to a neighborhood $\cU$ of the diagonal $\diag\subset M\times M$ using the exponential map $\exp_{x_0}$ centered at an arbitrary point $x_0\in M$, and then by  inspecting the dependence on $x_0$ one shows that this defines a family of distributions $\Fe{z}=  \pazocal{D}^\prime(\pazocal{U})$. 

In the general case when $P$ acts on sections of a vector bundle $E\to M$ with a connection $\nabla^E$, we pull-back the metric $g$ together with the bundle $E\to M$ and the connection $\nabla^E$, and use the same notation $g, E, \nabla^E$ for the pull-backed objects. We then trivialize the bundle by parallel transporting a fixed orthonormal frame spanning the fiber $E_0$ along rays emanating from $0$ (i.e.~along radial geodesics).  

The Hadamard parametrix of order $N$ is the distribution
$$
H_N(\lambda,.)=\sum_{k=0}^N \chi u_k \Fe[k]{\lambda}\in \pazocal{D}^\prime({M\times M;E\boxtimes E^*}),
$$
where $\chi\in \cf_{\rm c}(M\times M)$ is a cutoff function equal $1$ on a small neighborhood of the diagonal $\Delta$, and $u_k\in \cf(M\times M;E\boxtimes E^*)$ solve the hierarchy of  \emph{transport equations}: 
$$
2k u_k+hu_k+2\nabla_V^E u_k+2Pu_{k-1}=0, \quad k\in \nn_0,
$$
where $V$ is the Euler radial vector field, $h=V(\log|g|^{\12})$ is viewed as a diagonal multiplication operator (here $V$ acts as a Lie derivative), and by convention $u_{k-1}=0$ for $k=0$. Using the trivialization introduced above, the system of transport equations can be written in terms of $N\times N$ matrices. In the vector bundle case the connection $\nabla^S$ replaces the Lie derivative,  but other than that the same arguments as in the scalar case give existence of the $u_k$.

 Then,  a straightforward computation shows that $H_N$ solves  
\beq\label{eq:paraprox}
\left(P-\lambda\right)  H_N(\lambda,.)={ \module{g}^{-\frac{1}{2}}} \one_{\End(E)}\delta_{\diag}+(Pu_N)\mathbf{F}_N(\lambda,.)\chi+r_N(\lambda,.),
\eeq
where $r_N$ is an error term involving derivatives of $\chi$. The microlocal estimates on the error terms $(Pu_N)\mathbf{F}_N\chi$ and $r_N$ given in  \cite[Lem.~6.2]{Dang2020}  generalize immediately to the vector bundle case. In combination with the microlocal estimate from Step 1, they imply that it is possible to apply $(P-\lambda)^{-1}$ to both sides of \eqref{eq:paraprox} and to control microlocal properties of each term obtained by composition, uniformly in $\lambda$. Consequently,  
\beq\label{eq:reshad}
(P-\lambda)^{-1}=H_N(\lambda) + E_N(\lambda),
\eeq
 where   $E_N(\lambda,.)$  is regular as wanted near the diagonal, and decays at the desired rate along $Z_{\varepsilon}$ for  $N$  large enough (the precise statement is a direct  generalization of \cite[Prop.~6.3]{Dang2020}).  Next, one can insert \eqref{eq:reshad} into the identity
 $$
 (P-i\varepsilon)^{-{\cv}}=\frac{1}{2\pi i}\int_{Z_\varepsilon} (\lambda-i\varepsilon)^{-\cv} (P-\lambda)^{-1} d\lambda,
 $$
valid in the sense of $B(L^2(M;E))$ if $\Re \cv>0$. 
 Since 
 $$\frac{1}{2\pi i}\int_{Z_\varepsilon}(\lambda-i\varepsilon)^{-\cv} \mathbf{F}_\varm(\lambda,.) d\lambda =\frac{(-1)^\varm\Gamma(-\cv+1)}{\Gamma(-\cv-\varm+1)\Gamma(\cv+\varm)}   \mathbf{F}_{\varm+\cv-1}(i\varepsilon,.)$$
 and the integral of $E_N(\lambda,x,x)$ is holomorphic in $\cv$, the computation of the residues of $(P-i\varepsilon)^{-{\cv}}(x,x)$ is reduced to the analysis of  $\mathbf{F}_{\beta}(i\varepsilon,x,x)=F_\beta(i\varepsilon,0)$ and Gamma function factors, and to the computation of $u_k(x,x)$. This allows to conclude eventually
 \beq\label{eq:res}
 \res_{\alpha=\frac{n}{2}-k}(P-i\varepsilon)^{-\alpha}(x,x)= \frac{i u_k(x,x)}{2^n \pi^{\frac{n}{2}}(\frac{n}{2}-k-1)!}\in {\rm End}(S_x).
\eeq
as in the scalar case.

\step{3} Since the $\mathbf{F}_{\beta}(i\varepsilon,x,x)$ are scalar, only the computation of  $u_k(x,x)$ requires  some extra attention in the vector bundle case. We focus on the term $k=1$ relevant for the pole at $\cv=\frac{n}{2}-1$, and the term $k=0$ needed to initiate the recursion.  The $k=0$ transport equation 
reads 
$$ 
2\nabla^E_V u_0+h u_0=0, \quad u_0(0)=\one_{E_0},  $$
 hence $u_0(x)=\vert g(0)\vert^\frac{1}{4} \vert g(x)\vert^{-\frac{1}{4}} \one$ in our trivialization.
 
 The second transport equation is affected by the curvature term and reads
 $$
 \nabla_V^E u_1+\Big(1+\frac{h}{2}\Big)u_1=-Pu_0 $$
 and since both $\nabla_V^E u_1$ and $h$ vanish at $x=0$, this implies that
 $$
 u_1(0)=-Pu_0(0)=-P\big( \vert g(0)\vert^{\frac{1}{4}}\vert g(.)\vert^{-\frac{1}{4}}\one \big)(0).
 $$
 Using the expansion of the metric $g$ in normal coordinates and the Bochner--Lichnerowicz formula in the form \eqref{eq:BL2}, we find that
 $$
 -Pu_0(0)=\left(-\frac{R_g}{6}+\frac{R_g}{4}\right)\one+F^E.
 $$
 This implies that
 $$
 u_1(0)=\frac{R_g}{12}\one+F^E
 $$
 in normal coordinates. By combining this with \eqref{eq:res} we obtain the asserted result. 
\end{proof}

\begin{remark}\label{smallh}  Thanks to Theorem \ref{thm:final} (or strictly speaking, its direct analogue for $P+i \varepsilon$ instead of $P-i\varepsilon$, which is a somewhat more convenient choice in this paragraph) we can generalize a result from  \cite{Dang2020} about  functions $f(P+i\varepsilon)$ of $P$ for   $f$ Schwartz and such that  $\widehat{f}\in C^\infty_{\rm}(\open{0,+\infty})$. First, using our definition of  complex powers $(P+i\varepsilon)^{-\alpha}$ we can give  meaning to $f(P+i\varepsilon)$ by the  formula
$$f(P+i\varepsilon)=\frac{1}{2\pi i} \int_{\Re(\alpha)=c}e^{i\alpha\frac{\pi}{2}} (P+i\varepsilon)^{-\alpha} \Gamma(\alpha)\pazocal{M}\widehat{f}(\alpha)d\alpha, $$
where $\cM$ denotes the Mellin transform. The function $f(P+i\varepsilon) $ defined above  depends on the complex contour used to define complex powers of $P$, but for two such choices the difference is a smoothing operator of finite rank as discussed in Lemma~\ref{l:ambiguity_complex_powers}. Then, the proof from~\cite{Dang2020} applies verbatim to show that $f(P+i\varepsilon)$ has smooth Schwartz kernel on the diagonal and satisfies a small $h$ asymptotic of the form
$$
\bea
\tr_E\big( f( h (P+i\varepsilon))\big)(x,x)&=\frac{e^{in\frac{\pi}{4}}c_0}{i2^n\pi^{\frac{n}{2}}}h^{-n} \text{rk}(E) +
\frac{e^{i(n-2)\frac{\pi}{4}}c_0}{i2^n\pi^{\frac{n}{2}}}h^{-n+2}((-m^2-i\varepsilon)\text{rk}(E)\fantom +\text{rk}(E)\frac{R_g(x)}{12}+\tr(F^E)(x)
+{O}(h^{-n+4})
\eea
$$
where $c_0=\int_0^\infty\widehat{f}(t)t^{\frac{n}{2}-1}dt$, and further terms can be expressed in terms of Lorentzian analogues of heat kernel coefficients.
\end{remark}

{\small
\subsubsection*{Acknowledgments} 
The authors would like to thank the Erwin Schrödinger in Vienna for its hospitality during the programs  ``Spectral Theory and Mathematical Relativity'' and ``Nonlinear Waves and Relativity'' and the Institut Henri Poincaré for its hospitality during the program ``Quantum Fields Interacting with Geometry''. Support from the grant ANR-20-CE40-0018 of the Agence Nationale de la Recherche is gratefully acknowledged.  A.V.~gratefully acknowledges support from
the National Science Foundation under grant number DMS-2247004.  \medskip }

\bibliographystyle{abbrv}
\bibliography{complexpowers}

\begin{thebibliography}{10}

\bibitem{BGM}
C.~B{\"{a}}r, P.~Gauduchon, and A.~Moroianu.
\newblock {Generalized cylinders in semi-Riemannian and spin geometry}.
\newblock {\em Math. Zeitschrift}, 249(3):545--580, 2005.

\bibitem{Bar2019}
C.~B{\"{a}}r and A.~Strohmaier.
\newblock {An index theorem for Lorentzian manifolds with compact spacelike
  Cauchy boundary}.
\newblock {\em Am. J. Math.}, 141(5):1421--1455, 2019.

\bibitem{Baer2020}
C.~B{\"{a}}r and A.~Strohmaier.
\newblock {Local index theory for Lorentzian manifolds}.
\newblock {\em arXiv:2012.01364}, 2020.

\bibitem{BGV}
N.~Berline, E.~Getzler, and M.~Vergne.
\newblock {\em {Heat Kernels and Dirac Operators}}.
\newblock 2004.

\bibitem{Chamseddine1997}
A.~H. Chamseddine and A.~Connes.
\newblock {The spectral action principle}.
\newblock {\em Commun. Math. Phys.}, 1997.

\bibitem{Chamseddine2007}
A.~H. Chamseddine, A.~Connes, and M.~Marcolli.
\newblock {Gravity and the standard model with neutrino mixing}.
\newblock {\em Adv. Theor. Math. Phys.}, 11(6):991--1089, 2007.

\bibitem{Chihara2002}
H.~Chihara.
\newblock {Smoothing effects of dispersive pseudodifferential equations}.
\newblock {\em Commun. Partial Differ. Equations}, 27(9-10):1953--2005, 2002.

\bibitem{cdv}
Y.~{Colin de Verdi{\`{e}}re} and C.~{Le Bihan}.
\newblock {On essential-selfadjointness of differential operators on closed
  manifolds}.
\newblock {\em Ann. la Fac. des Sci. Toulouse Math{\'{e}}matiques},
  31(5):1287--1302, 2022.

\bibitem{Connes1996}
A.~Connes.
\newblock {Gravity coupled with matter and the foundation of non-commutative
  geometry}.
\newblock {\em Commun. Math. Phys.}, 182(1):155--176, 1996.

\bibitem{Connes2008a}
A.~Connes and M.~Marcolli.
\newblock {\em {Noncommutative Geometry, Quantum Fields and Motives}}.
\newblock American Mathematical Society, Providence, RI, 2008.

\bibitem{Damaschke}
O.~Damaschke.
\newblock {Cauchy Problem for the Dirac operator on spatially non-compact
  spacetimes}.
\newblock {\em arXiv:2409.17344}, 2024.

\bibitem{DAndrea2016}
F.~D'Andrea, M.~A. Kurkov, and F.~Lizzi.
\newblock {Wick rotation and fermion doubling in noncommutative geometry}.
\newblock {\em Phys. Rev. D}, 94(2):025030, 2016.

\bibitem{Dang2021}
N.~V. Dang and M.~Wrochna.
\newblock {Dynamical residues of Lorentzian spectral zeta functions}.
\newblock {\em J. l'{\'{E}}cole Polytech. — Math{\'{e}}matiques},
  9:1245--1292, 2022.

\bibitem{Dang2020}
N.~V. Dang and M.~Wrochna.
\newblock {Complex powers of the wave operator and the spectral action on
  Lorentzian scattering spaces}.
\newblock {\em J. Eur. Math. Soc.}, 2023.

\bibitem{Dang2022}
N.~V. Dang and M.~Wrochna.
\newblock {Lorentzian Spectral Zeta Functions on Asymptotically Minkowski
  Spacetimes}.
\newblock In J.~{K{\"{a}}hler, U., Reissig, M., Sabadini, I., Vindas}, editor,
  {\em Anal. Appl. Comput. ISAAC 2021.}, pages 501--514. Birkh{\"{a}}user,
  2023.

\bibitem{Derezinski2024}
J.~Derezi{\'{n}}ski and C.~Ga{\ss}.
\newblock {Propagators in curved spacetimes from operator theory}.
\newblock {\em arXiv:2409.03279}, 2024.

\bibitem{derezinski}
J.~Derezi{\'{n}}ski and D.~Siemssen.
\newblock {Feynman propagators on static spacetimes}.
\newblock {\em Rev. Math. Phys.}, 2018.

\bibitem{Devastato2018}
A.~Devastato, S.~Farnsworth, F.~Lizzi, and P.~Martinetti.
\newblock {Lorentz signature and twisted spectral triples}.
\newblock {\em J. High Energy Phys.}, 2018(3):89, 2018.

\bibitem{GHV}
J.~Gell-Redman, N.~Haber, and A.~Vasy.
\newblock {The Feynman propagator on perturbations of Minkowski space}.
\newblock {\em Commun. Math. Phys.}, 342(1):333--384, 2016.

\bibitem{hassell}
J.~Gell-Redman, A.~Hassell, J.~Shapiro, and J.~Zhang.
\newblock {Existence and Asymptotics of Nonlinear Helmholtz Eigenfunctions}.
\newblock {\em SIAM J. Math. Anal.}, 52(6):6180--6221, 2020.

\bibitem{GWfeynman}
C.~G{\'{e}}rard and M.~Wrochna.
\newblock {The massive Feynman propagator on asymptotically Minkowski
  spacetimes}.
\newblock {\em Am. J. Math.}, 141(6):1501--1546, 2019.

\bibitem{Gerard2019b}
C.~G{\'{e}}rard and M.~Wrochna.
\newblock {The massive Feynman propagator on asymptotically Minkowski
  spacetimes II}.
\newblock {\em Int. Math. Res. Not.}, 2019.

\bibitem{JMS}
Q.~Jia, M.~Molodyk, and E.~Sussman.
\newblock {The essential self-adjointness of the wave operator on radiative
  spacetimes}.
\newblock {\em arXiv:2412.03828}, 2024.

\bibitem{kaminski}
W.~Kami{\'{n}}ski.
\newblock {Non-Self-Adjointness of the Klein–Gordon Operator on a Globally
  Hyperbolic and Geodesically Complete Manifold: An Example}.
\newblock {\em Ann. Henri Poincar{\'{e}}}, 23(12):4409--4427, 2022.

\bibitem{Kottke2015}
C.~Kottke.
\newblock {A Callias-type index theorem with degenerate potentials}.
\newblock {\em Commun. Partial Differ. Equations}, 40(2):219--264, 2015.

\bibitem{lawsonspin}
H.~B. Lawson and M.-L. Michelsohn.
\newblock {\em {Spin Geometry}}, volume~38.
\newblock Princeton University Press, Princeton, NJ, 2016.

\bibitem{martinetti}
P.~Martinetti and D.~Singh.
\newblock {Lorentzian fermionic action by twisting Euclidean spectral triples}.
\newblock {\em J. Noncommutative Geom.}, 16(2):513--559, 2022.

\bibitem{Melrose1993}
R.~Melrose.
\newblock {\em {The Atiyah--Patodi--Singer Index Theorem}}.
\newblock CRC Press, Boca Raton, FL, 1993.

\bibitem{melrosered}
R.~Melrose.
\newblock {Spectral and scattering theory for the Laplacian on asymptotically
  Euclidian spaces}.
\newblock In {\em Spectr. Scatt. Theory}. Dekker, New York, 1994.

\bibitem{MolodhykVasy}
M.~Molodyk and A.~Vasy.
\newblock {An analogue of non-interacting quantum field theory in Riemannian
  signature}.
\newblock {\em arXiv:2404.11821}, 2024.

\bibitem{Moretti2003}
V.~Moretti.
\newblock {Aspects of noncommutative Lorentzian geometry for globally
  hyperbolic spacetimes}.
\newblock {\em Rev. Math. Phys.}, 15(10):1171--1217, 2003.

\bibitem{nakamurataira}
S.~Nakamura and K.~Taira.
\newblock {Essential self-adjointness of real principal type operators}.
\newblock {\em Ann. Henri Lebesgue}, 4:1035--1059, 2021.

\bibitem{Nakamura2022}
S.~Nakamura and K.~Taira.
\newblock {A Remark on the Essential Self-adjointness for Klein–Gordon-Type
  Operators}.
\newblock {\em Ann. Henri Poincar{\'{e}}}, 24(8):2587--2605, 2023.

\bibitem{Nakamura2022a}
S.~Nakamura and K.~Taira.
\newblock {Essential self-adjointness of Klein-Gordon type operators on
  asymptotically static, Cauchy-compact spacetimes}.
\newblock {\em Commun. Math. Phys.}, 398(3):1153--1169, 2023.

\bibitem{Paschke2004}
M.~Paschke and R.~Verch.
\newblock {Local covariant quantum field theory over spectral geometries}.
\newblock {\em Class. Quantum Gravity}, 21(23):5299--5316, 2004.

\bibitem{Roe}
J.~Roe.
\newblock {\em {Elliptic Operators, Topology, and Asymptotic Methods}}.
\newblock Chapman \& Hall, 1999.

\bibitem{Strohmaier2006}
A.~Strohmaier.
\newblock {On noncommutative and pseudo-Riemannian geometry}.
\newblock {\em J. Geom. Phys.}, 56(2):175--195, 2006.

\bibitem{Sussman2023a}
E.~Sussman.
\newblock {Massive wave propagation near null infinity}.
\newblock {\em arXiv:2305.01119}, 2023.

\bibitem{Tadano2019}
Y.~Tadano and K.~Taira.
\newblock {Uniform bounds of discrete Birman–Schwinger operators}.
\newblock {\em Trans. Am. Math. Soc.}, 372(7):5243--5262, 2019.

\bibitem{taira}
K.~Taira.
\newblock {Equivalence of classical and quantum completeness for real principal
  type operators on the circle}.
\newblock {\em arXiv:2004.07547}, 2020.

\bibitem{Taira2020a}
K.~Taira.
\newblock {Limiting absorption principle and equivalence of Feynman propagators
  on asymptotically Minkowski spacetimes}.
\newblock {\em Commun. Math. Phys.}, 388(1):625--655, 2021.

\bibitem{Taira2022}
K.~Taira.
\newblock {Remarks on the geodesically completeness and the smoothing effect on
  asymptotically Minkowski spacetimes}.
\newblock {\em Lett. Math. Phys.}, 112(2):22, 2022.

\bibitem{Uhlmann2016}
G.~Uhlmann and A.~Vasy.
\newblock {The inverse problem for the local geodesic ray transform}.
\newblock {\em Invent. Math.}, 205(1):83--120, 2016.

\bibitem{VanDenDungen2018}
K.~van~den Dungen.
\newblock {Families of spectral triples and foliations of space(time)}.
\newblock {\em J. Math. Phys.}, 59(6):063507, 2018.

\bibitem{VanSuijlekom2004}
W.~D. van Suijlekom.
\newblock {The noncommutative Lorentzian cylinder as an isospectral
  deformation}.
\newblock {\em J. Math. Phys.}, 45(1):537--556, 2004.

\bibitem{vasygrenoble}
A.~Vasy.
\newblock {A Minicourse on Microlocal Analysis for Wave Propagation}.
\newblock In {\em Asymptot. Anal. Gen. Relativ.}, pages 219--374. Cambridge
  University Press, 2017.

\bibitem{Vasy2017b}
A.~Vasy.
\newblock {On the positivity of propagator differences}.
\newblock {\em Ann. Henri Poincar{\'{e}}}, 18(3):983--1007, 2017.

\bibitem{vasyessential}
A.~Vasy.
\newblock {Essential self-adjointness of the wave operator and the limiting
  absorption principle on Lorentzian scattering spaces}.
\newblock {\em J. Spectr. Theory}, 10(2):439--461, 2020.

\bibitem{vasylap}
A.~Vasy.
\newblock {Limiting absorption principle on Riemannian scattering
  (asymptotically conic) spaces, a Lagrangian approach}.
\newblock {\em Commun. Partial Differ. Equations}, 46(5):780--822, 2021.

\bibitem{vasyresolvent}
A.~Vasy.
\newblock {Resolvent near zero energy on Riemannian scattering (asymptotically
  conic) spaces}.
\newblock {\em Pure Appl. Anal.}, 3(1):1--74, 2021.

\bibitem{Vasy:Semiclassical-standard}
A.~Vasy.
\newblock {On the relationship between the semiclassical and standard
  pseudodifferential algebras}.
\newblock 2024.

\bibitem{vasywrochna}
A.~Vasy and M.~Wrochna.
\newblock {Quantum fields from global propagators on asymptotically Minkowski
  and extended de Sitter spacetimes}.
\newblock {\em Ann. Henri Poincar{\'{e}}}, 2018.

\bibitem{Vasy2000}
A.~Vasy and M.~Zworski.
\newblock {Semiclassical estimates in asymptotically Euclidean scattering}.
\newblock {\em Commun. Math. Phys.}, 212(1):205--217, 2000.

\bibitem{Wrochna2022}
M.~Wrochna and R.~Zeitoun.
\newblock {The wave resolvent for compactly supported perturbations of
  Minkowski space}.
\newblock In {\em Harmon. Anal. Partial Differ. Equations}, pages 1--17. 2022.

\bibitem{Wunsch1999a}
J.~Wunsch.
\newblock {Propagation of singularities and growth for Schr{\"{o}}dinger
  operators}.
\newblock {\em Duke Math. J.}, 98(1), 1999.

\end{thebibliography}

\end{document}